\documentclass[12pt,a4paper,reqno]{amsart}
\usepackage{amssymb}
\usepackage{amscd}
\usepackage{enumerate}
\usepackage{graphicx}
\usepackage{siunitx}
\usepackage{tikz-cd}
\usepackage{bm}
\numberwithin{equation}{section}

\usepackage{mathtools}
\usepackage[tableposition=top]{caption}
\usepackage{booktabs,dcolumn}

\theoremstyle{plain}

\newtheorem{theorem}{Theorem}[section]
\newtheorem{proposition}[theorem]{Proposition}
\newtheorem{lemma}[theorem]{Lemma}
\newtheorem{corollary}[theorem]{Corollary}

\theoremstyle{definition}

\newtheorem{definition}[theorem]{Definition}

\newtheorem{example}[theorem]{Example}

\renewcommand\P{\mathbb{P}}
\newcommand\E{\mathbb{E}}
\newcommand\R{\mathbb{R}}
\newcommand\Z{\mathbb{Z}}
\newcommand\N{\mathbb{N}}
\newcommand\A{\mathbb{A}}
\newcommand\T{\mathbb{T}}
\newcommand\Q{\mathbb{Q}}
\newcommand\C{\mathbb{C}}
\newcommand\G{\mathbb{G}}
\newcommand\HH{\mathbb{H}}
\newcommand\F{\mathcal{F}}

\newcommand\Sample{\mathcal{S}}
\newcommand\Schwartz{\mathbf{S}}

\newcommand\Op{{\operatorname{T}}}
\newcommand\Sun{{\operatorname{Sun}}}
\newcommand\Log{{\operatorname{Log}}}
\newcommand\eps{\varepsilon}

\renewcommand{\mod}{\bmod}

\begin{document}

\title[Ionescu--Wainger and the adeles]{The Ionescu--Wainger multiplier theorem and the adeles}

\author{Terence Tao}
\address{UCLA Department of Mathematics, Los Angeles, CA 90095-1555.}
\email{tao@math.ucla.edu}


\subjclass[2010]{42B15}

\begin{abstract}  The Ionescu--Wainger multiplier theorem establishes good $L^p$ bounds for Fourier multiplier operators localized to major arcs; it has become an indispensible tool in discrete harmonic analysis.  We give a simplified proof of this theorem with more explicit constants (removing logarithmic losses that were present in previous versions of the theorem), and give a more general variant involving adelic Fourier multipliers.  We also establish a closely related adelic sampling theorem that shows that $\ell^p(\Z^d)$ norms of functions with Fourier transform supported on major arcs are comparable to the $L^p(\A_\Z^d)$ norm of their adelic counterparts.
\end{abstract}

\maketitle


\section{Introduction}

This paper will be concerned with the $L^p$ theory of Fourier multiplier operators on various locally compact abelian groups, such as $\Z^d$, $\R^d$, and $\A_\Z^d$.  In order to treat these groups in a unified fashion we adopt the following abstract harmonic analysis notation.

\begin{definition}[Pontryagin duality]  An \emph{LCA group} is a locally compact abelian group $\G = (\G,+)$ equipped with a Haar measure $\mu_{\G}$.  A \emph{Pontryagin dual} of an LCA group $\G$ is an LCA group $\G^* = (\G^*,+)$ with a Haar measure $\mu_{\G^*}$ and a continuous bihomomorphism $(x,\xi) \mapsto x \cdot \xi$ (which we call a \emph{pairing}) from $\G \times \G^*$ to the unit circle $\T = \R/\Z$, such that the Fourier transform $\F_{\G} \colon L^1(\G) \to C(\G^*)$ defined by
$$ \F_{\G} f(\xi) \coloneqq \int_{\G} f(x) e(x \cdot \xi)\ d\mu_{\G}(x),$$
where $e \colon \T \to \C$ is the standard character $e(\theta) \coloneqq e^{2\pi i \theta}$, extends to a unitary map from $L^2(\G)$ to $L^2(\G^*)$; in particular we have the Plancherel identity
$$ \int_\G |f(x)|^2\ d\mu_\G(x) = \int_{\G^*} |\F_\G f(\xi)|^2\ d\mu_{\G^*}(\xi)$$
for all $f \in L^2(\G)$, as well as the inversion formula
$$ \F_{\G}^{-1} F(x) = \int_{\G^*} F(\xi) e(-x \cdot \xi)\ d\mu_{\G^*}(\xi)$$
for all $F \in L^1(\G^*) \cap L^2(\G^*)$.

If $\Omega \subset \G^*$ is measurable, we say that $f \in L^2(\G)$ is \emph{Fourier supported} in $\Omega$ if $\F_{\G} f$ vanishes outside of $\Omega$ (modulo null sets).  The space of such functions will be denoted $L^2(\G)^\Omega$.

If $m \in L^\infty(\G^*)$, we define the associated Fourier multiplier operator $\Op_m \colon L^2(\G) \to L^2(\G)$ by the formula
$$ \F_\G \Op_m f \coloneqq m \F_\G f $$
for all $f \in L^2(\G)$, thus $\Op_m = \F_G^{-1} m \F_G$. We refer to $m$ as the \emph{symbol} of $\Op_m$.

For any finite-dimensional normed vector space $V$, we extend $\Op_m$ to an operator on $L^2(\G;V)$ in the obvious fashion.
\end{definition}

We will focus in particular on the Pontryagin dual pairs
$$ (\G, \G^*) = (\Z^d, \T^d), (\R^d, \R^d), (\A_\Z^d, \R^d \times (\Q/\Z)^d)$$
where $\A_\Z = \R \times \hat \Z$ denotes the adelic integers and $d \geq 1$ is an integer; see Appendix \ref{adele-sec} for a more precise description of these pairs.  To avoid technicalities we shall largely restrict attention to smooth symbols $m$, although rougher symbols can also be treated by applying suitable limiting arguments, as our estimates will not depend on any smooth norms of $m$.  We view the adelic space $\A_\Z^d = \R^d \times \hat \Z^d$ as a simplified model of the lattice $\Z^d$ that captures both the ``continuous'' aspects of this lattice (via the factor $\R^d$) and the ``arithmetic'' aspects of this lattice (via the factor $\hat \Z^d$).  The reader may wish to restrict attention to the one-dimensional case $d=1$ as it already captures all of the key ideas, but the extension to higher values of $d$ requires only minor notational changes and is also useful in some applications (e.g., \cite{MSZ3}), so we work with general $d$ in this paper.

A central problem in harmonic analysis is to understand the operator norm $\| \Op_m \|_{B(L^p(\G))}$ of a Fourier multiplier operator $\Op_m$ on a Lebesgue space $L^p(\G)$ (restricting $\Op_m$ initially to some dense subclass such as the Schwartz-Bruhat space $\Schwartz(\G)$ to avoid technicalities).  For $p=2$ this norm is just the $L^\infty(\G^*)$ norm of $m$, but for other choices of $p$ the situation is considerably more complicated. Our initial focus here will be on understanding this problem in the case where $\G = \Z^d$ and $m$ is supported on ``major arcs''.

If $m \in C^\infty_c(\R^d)$ is a smooth symbol, then $\Op_m$ is a Fourier multiplier operator on $L^2(\R^d)$, but we can also define associated Fourier multiplier operators $\Op_{m;\alpha}$ on $\ell^2(\Z^d)$ for various shifts $\alpha \in \T^d$ by the formula
$$ \Op_{m,\alpha} \coloneqq \Op_{m_\alpha}$$
where $m_\alpha \in C^\infty(\T^d)$ is the symbol
$$ m_\alpha(\xi) \coloneqq \sum_{\theta \in \R^d: \xi = \alpha + \theta \mod \Z^d} m(\theta).$$
Equivalently, one has
$$ \Op_{m;\alpha} f(n) = \int_{\R^d} m(\theta) e( -n \cdot (\alpha + \theta) ) \F_{\Z^d} f(\alpha + \theta)\ d\theta.$$
More generally, for any finite set $\Sigma \subset \R^d$, define
\begin{equation}\label{msig}
 \Op_{m;\Sigma} \coloneqq \sum_{\alpha \in \Sigma} \Op_{m;\alpha},
\end{equation}
thus
$$ \Op_{m;\Sigma} f(n) = \sum_{\alpha \in \Sigma} \int_{\R^d} m(\theta) e( -n \cdot (\alpha + \theta) ) \F_{\Z^d} f(\alpha + \theta)\ d\theta.$$

If the support of $m$ is suitably restricted, then the $\ell^p(\R^d)$ multiplier theory of $\Op_{m;\alpha}$ or $\Op_{m;\Sigma}$ is closely tied to the $L^p(\Z^d)$ multiplier theory of $\Op_m$.  One basic manifestation of this is via the following sampling principle of Maygar, Stein, and Wainger \cite{msw}.  For any $\xi_0 \in \R^d$, let $\xi_0 + [-r,r]^d$ denote the closed cube of sidelength $2r$ centred at $\xi_0$.  We also define analogous balls (or cubes or ``arcs'') $\alpha + [-r,r]^d \subset \T^d$ for $\alpha \in \T^d$. For any positive integer $Q$, let
$$ \T^d[Q] \coloneqq \{ x \in \T^d: Qx = 0 \} = \left( \frac{1}{Q}\Z/\Z\right)^d$$
denote the collection of $Q$-torsion points of the torus $\T^d$; this is a finite subgroup of $\T^d$.  

\begin{proposition}[Maygar--Stein--Wainger sampling principle]\label{msw}  Let $d \geq 1$ be an integer, let $1 \leq p \leq \infty$, and let $V$ be a finite-dimensional Banach space.
\begin{itemize}
\item[(i)]  If $m \in C^\infty_c(\R^d)$ is supported in $[-\frac{1}{2},\frac{1}{2}]^d$, then 
\begin{equation}\label{opm}
\| \Op_{m;0} \|_{B(\ell^p(\Z^d;V))} \leq O(1)^d \| \Op_{m} \|_{B(L^p(\R^d;V))}.
\end{equation}
(See Section \ref{notation-sec} for our conventions on asymptotic notation, as well as our notation $B(W)$ for the operator norm on a normed vector space $W$.)
\item[(ii)]  More generally, if $Q \geq 1$ is an integer, and $m \in C^\infty_c(\R^d)$ is supported in $[-\frac{1}{2Q},\frac{1}{2Q}]^d$, then
$$ \| \Op_{m;\T^d[Q]} \|_{B(\ell^p(\Z^d;V))} \leq O(1)^d \| \Op_{m} \|_{B(L^p(\R^d;V))}.$$
\end{itemize}
\end{proposition}

\begin{proof} Part (i) is \cite[Proposition 2.1]{msw} and part (ii) is \cite[Corollary 2.1]{msw}, after noting that all implied constants in the proof are at most exponential in the dimension $d$.
\end{proof}

In the remarks after \cite[Proposition 2.1]{msw} the question is posed as to whether the $O(1)^d$ constant in \eqref{opm} can be made independent of $d$, or even replaced with $1$.  This is easily seen to be true for $p=2$ from Plancherel's theorem, but for $p$ sufficiently close to $1$ or $\infty$ the answer is negative \cite{K}.  The question for other values of $p$ remains open.

Proposition \ref{msw} lets us control certain Fourier multiplier operators whose symbol is supported in sets of the form
$$ \T^d[Q] + [-\eps,\eps]^d$$
for $\eps>0$ small enough (in particular, the above proposition applies when $\eps \leq 1/2Q$).  For applications to discrete harmonic analysis (particularly involving averaging over ``arithmetic'' sets such as polynomial sequences or primes), it would be desirable to have a similar estimate that could handle symbols supported on the ``classical major arcs''
\begin{equation}\label{major-class}
 \bigcup_{q=1}^N \T^d[q] + [-\eps,\eps]^d
\end{equation}
for some $N \geq 1$ and $\eps>0$.  As it turns out, these classical arcs are inconvenient to work with directly for the purposes of $\ell^p$ multiplier theory.  However, in the remarkable work of Ionescu and Wainger \cite{IW}, a more complicated major arc set 
$$ {\mathcal M} = \Sigma_{\leq k} + [-\eps,\eps]^d$$
was introduced which (for suitable choices of parameters) contained the classical major arc set \eqref{major-class} while simultaneously enjoying a satisfactory $\ell^p$ multiplier theory for relatively large values of $\eps$.  The Ionescu--Wainger multiplier theorem has since been indispensable in many results in discrete harmonic analysis, in particular providing an analogue of Littlewood-Paley theory adapted to major arcs; see e.g., \cite{pierce0}, \cite{ISMW}, \cite{MST0}, \cite{MST}, \cite{MSZ3}, \cite{KMT}.

In this note we give a general version of the Ionescu--Wainger theorem which avoids some logarithmic loss factors present in earlier treatments, and quantifies the dependence on various parameters.  To describe the result we need some notation.

\begin{definition}[Generalized Ionescu--Wainger major arcs]\label{major-arc}  A \emph{major arc parameter set} is a quadruplet $(d,k,S,\eps)$ where $d,k \geq 1$ are integers, 
$S$ is a finite collection of pairwise coprime integers, and $\eps > 0$ is a real number.
For any $A \subseteq S$, we let $Q_A$ denote the positive integer $Q_A \coloneqq \prod_{q \in A} q$ (with the usual convention $Q_\emptyset = 1$), and let $\Sigma_{\subseteq A} \subseteq \T^d$ denote the subgroup 
$$\Sigma_{\subseteq A} \coloneqq \T^d[Q_A].$$
We let $\Sigma_A$ denote the set
$$ \Sigma_A \coloneqq \Sigma_{\subseteq A} \backslash \bigcup_{B \subsetneq A} \Sigma_{\subseteq B}.$$
Thus for instance $\Sigma_\emptyset = \Sigma_{\subseteq \emptyset} = \T^d[1] = \{0\}$.
We define the sets
$$ \binom{S}{k} \coloneqq \{ A \subseteq S: |A| = k \}; \quad \binom{S}{\leq k} \coloneqq \{ A \subseteq S: 0 \leq |A| \leq k \}$$
and let $\Sigma_{\leq k}$ denote the set
$$ \Sigma_{\leq k} \coloneqq \bigcup_{A \in \binom{S}{\leq k}} \Sigma_{\subseteq A} = \bigcup_{A \in \binom{S}{\leq k}} \Sigma_{A}.$$
The \emph{major arc set} ${\mathcal M} = {\mathcal M}_{(d,k,S,\eps)}$ associated to the parameter set $(d,k,S,\eps)$ is defined as
$$ {\mathcal M} \coloneqq \Sigma_{\leq k} + [-\eps,\eps]^d.$$

A major arc parameter set is said to be \emph{$(r,c)$-good} for some integer $r \geq 1$ and $0 < c < 1$ if one has the smallness condition
\begin{equation}\label{small}
 \eps < \frac{c}{2r q_{\max}^{2rk}}
\end{equation}
for some integer $q_{\max}$ that is greater than or equal to all the elements of $S$.
\end{definition}

Expanding out the definitions, we see that ${\mathcal M}$ consists of all elements of $\T^d$ of the form $\frac{a}{q} + \theta \mod \Z^d$, where $q$ is the product of at most $k$ elements of $S$, $a \in \Z^d$, and $\theta \in [-\eps,\eps]^d$.  The major arc sets ${\mathcal M}$ considered here are more general than the ones constructed in \cite{IW}, which involved a specific choice of $S$ involving a partition of all the primes up to a certain threshold.  In Section \ref{num-sec} we explain how the major arcs in \cite{IW} become a special case of the ones considered here. However, as we shall see, the exact structure of the set $S$ plays almost no role in the Ionescu--Wainger theory, other than via the upper bound $q_{\max}$ on the elements of $S$. The parameter $c$ is of minor technical importance and the reader may wish to fix it as an absolute constant (e.g., $c=1/2$) for most of the following discussion. In typical applications one should think of the quantities $d,k,r$ as being bounded, $|S|$ and $q_{\max}$ as being large, and $\eps$ as being quite small.

We can now state our first form of the Ionescu--Wainger multiplier theorem.  To simplify the bounds slightly we adopt the notation
$$ \Log x \coloneqq \log(2+x).$$

\begin{theorem}[Ionescu--Wainger multiplier theorem, real form]\label{main}  Let $(d,k,S,\eps)$ be a major arc parameter set, and let $H$ be a finite-dimensional Hilbert space.  Let $m \in C^\infty_c(\R^d)$ be supported on $[-\eps,\eps]^d$.  Then if $(d,k,S,\eps)$ is $(r,c)$-good for some integer $r \geq 1$ and $0 < c < 1$, one has
$$
\| \Op_{m;\Sigma_{\leq k}} \|_{B(\ell^{2r}(\Z^d;H))}
\leq O_c(1)^d O(r \Log^{1/2}(k))^k \|\Op_m\|_{B(L^{2r}(\R^d))};$$
more generally, one has
\begin{align*}
\| \sum_{A \in \binom{S}{\leq k}} \epsilon_A \Op_{m;\Sigma_A} \|_{B(\ell^{2r}(\Z^d;H))}
 \leq O_c(1)^d O(r \Log^{1/2}(k))^k \|\Op_m\|_{B(L^{2r}(\R^d))}.
\end{align*}
whenever $\epsilon_A$ is a complex number with $|\epsilon_A| \leq 1$ for each $A \in \binom{S}{\leq k}$.
\end{theorem}

The factor of $O_c(1)^d O(r \Log^{1/2}(k))^k$ looks somewhat messy, but the key point is that it is uniform in the parameters $S, \eps, H, m$; in particular it can be simplified to $O_{c,d,k,r}(1)$.  The original version of this result in \cite{IW}, when adapted to this notion of major arc, gave instead a bound of the form $O_{c,d,k,r}(\Log^k |S|)$, which was later refined in \cite{mirek} to $O_{c,d,k,r}( \Log |S| )$ (see also \cite{MST0}, \cite{MST}, \cite{MSZ3}).  The dependence of $c$ will be unimportant in applications as one can typically take $c$ to equal a constant value such as $c=1/2$.  The restriction to even integer exponents $2r$ will be removed in Theorem \ref{main-adelic} below (at the cost of worsening the bounds slightly).  We work with finite dimensional Hilbert spaces $H$ here to avoid technical complications, but one can extend this result to separable Hilbert spaces without difficulty by a standard limiting argument.

We prove Theorem \ref{main} in Section \ref{proofs-sec}, after some preliminaries in Sections \ref{super-sec}, \ref{sunflower-sec}.  The argument uses the same basic approach as previous proofs of the Ionescu--Wainger theorem in the literature (particularly \cite[Theorem 2.1]{MSZ3}), which we summarize as follows.

\begin{itemize}
\item[(i)]  One begins by exploiting ``Type II superorthogonality'' (following the terminology recently introduced by Pierce \cite{pierce}) of the terms $\Op_{m;\Sigma_A} f$ (arising from ``denominator orthogonality'' of the rational set $\Sigma_{\leq k}$) to estimate the $\ell^{2r}(\Z^d;H)$ norm of $\sum_{A \in \binom{S}{\leq k}} \epsilon_A \Op_{m;\Sigma_A} f$ by the $\ell^{2r}(\Z^d; H^{[\leq k]})$ of the square function $(\Op_{m;\Sigma_A} f)_{A \in \binom{S}{\leq k}}$ that takes values in $H^{[\leq k]} \coloneqq H^{\binom{S}{\leq k}}$, in the spirit of reverse square function estimates of Khintchine type.
\item[(ii)]  Using a ``nonconcentration estimate'', one estimates this square function norm by an expression summing over various ``sunflowers'' in $S$.
\item[(iii)]  By exploiting ``numerator orthogonality'' of the rational set $\Sigma_{\leq k}$, one estimates this resulting sum over sunflowers by a more tractable square function involving the functions $\Op_{m;\alpha + \Sigma_{\subseteq A_0}} f$ that are summed over cosets of a fixed finite subgroup $\Sigma_{\subseteq A_0} = \T^d[Q_{A_0}]$ of $\T^d$.
\item[(iv)]  At this point the symbol $m$ can be disposed of using the Marcinkiewicz--Zygmund theorem, and then the resulting quantity can be estimated using a square function estimate of Rubio de Francia type \cite{rubio}.
\end{itemize}

Our main innovations are to eliminate logarithmic losses in (i) using the probabilistic decoupling trick (cf. \cite{pena}), and to obtain efficient bounds in (ii) by using recent progress \cite{rao} on the sunflower conjecture of Erd\H{o}s and Rado \cite{erdos}.

We also interpret these results through the lens of adelic harmonic analysis, following the slogan
$$ \hbox{Major arc analysis on } \Z^d \approx \hbox{Low frequency analysis on } \A_\Z^d$$
recently advocated (in the one-dimensional setting $d=1$) in \cite{KMT}.  As reviewed in Appendix \ref{adele-sec}, we have an inclusion homomorphism
$$ \iota \colon \Z^d \to \A_\Z^d$$
and an addition homomorphism
$$ \pi \colon \R^d \times (\Q/\Z)^d \to \T^d$$
that is Fourier adjoint to $\iota$, with $\pi$ being given explicitly by
$$ \pi( \theta, \alpha ) \coloneqq \alpha + \theta$$
for $\theta \in \R^d$ and $\alpha \in (\Q/\Z)^d$.

There is a \emph{sampling map} $\Sample \colon \Schwartz(\A_\Z^d) \to \Schwartz(\Z^d)$, where $\Schwartz(\G)$ denotes the Schwartz-Bruhat space on $\G$ (as defined in Appendix \ref{adele-sec}), defined by
$$ \Sample f \coloneqq f \circ \iota.$$
As in \cite{KMT}, we say that a compact subset $\Omega$ of adelic frequency space $\R^d \times (\Q/\Z)^d$ is \emph{non-aliasing} if the projection map $\pi$ is injective on $\Omega$.  In \cite[(4.6)]{KMT} it was shown that the sampling map extends to a unitary transformation
$$ \Sample \colon L^2(\A_\Z^d)^\Omega \to \ell^2(\Z^d)^{\pi(\Omega)}$$
for any non-aliasing compact $\Omega \subset \R^d \times (\Q/\Z)^d$ (the argument was presented for $d=1$, but extends to arbitrary dimension), where we recall that $L^2(\A_\Z^d)^\Omega$ denotes the closed subspace of $L^2(\A_\Z^d)$ consisting of functions whose Fourier transform is supported in $\Omega$, and similarly for $\ell^2(\Z^d)^{\pi(\Omega)}$.  In particular one has a unitary \emph{interpolation map}
$$ \Sample_\Omega^{-1} \colon \ell^2(\Z^d)^{\pi(\Omega)} \to L^2(\A_\Z^d)^\Omega$$
that inverts $\Sample$; we extend these operators to vector-valued functions taking values in a finite-dimensional vector space in the obvious fashion.  For instance, if $\Omega$ is a set of the form $[-\eps,\eps]^d \times \Sigma$ for some $\eps > 0$ and some finite $\Sigma \subset (\Q/\Z)^d$, then $\Omega$ is non-aliasing if the elements of $\Sigma$ are separated from each other by more than $2\eps$ (in the $\ell^\infty$ metric), and then $\pi(\Omega) = \Sigma + [-\eps,\eps]^d$ and every element $f$ of $\ell^2(\Z^d)^{\pi(\Omega)}$ then has a unique Fourier representation of the form
$$ f(n) = \sum_{\alpha \in \Sigma} \int_{[-\eps,\eps]^d} e(-n \cdot (\alpha+\theta)) \F_{\Z^d} f(\alpha+\theta)\ d\theta$$
for $n \in \Z^d$, and the interpolated function $\Sample_\Omega^{-1} f \in L^2(\A_\Z^d)^\Omega$ is then given by the formula
$$ \Sample_\Omega^{-1} f (x,y) = \sum_{\alpha \in \Sigma} \int_{[-\eps,\eps]^d} e(-x \cdot \theta - y \cdot \alpha) \F_{\Z^d} f(\alpha+\theta)\ d\theta$$
and then it is clear that $f = \Sample \Sample_\Omega^{-1} f$.

From unitarity we have
$$ \| \Sample_\Omega^{-1} f \|_{L^2(\A_\Z^d)} = \|f\|_{\ell^2(\Z^d)}$$
whenever $f \in \ell^2(\Z^d)^{\pi(\Omega)}$.  In many cases we can extend this $L^2$ isometry to an $L^p$ equivalence.  For instance, we have

\begin{proposition}[Quantitative Shannon sampling theorem]\label{qss} Let $1 \leq p \leq \infty$, and $V$ be a finite-dimensional normed vector space.  If $\Omega$ is the (non-aliasing) set $\Omega \coloneqq [-\frac{c}{Q},\frac{c}{Q}]^d \times \T^d[Q]$ for some positive integer $Q$ and $0 < c < \frac{1}{2}$ then $\Sample_\Omega$ extends to a bounded invertible linear operator from $\ell^p(\Z^d;V)^{\pi(\Omega)}$ to $L^p(\Z^d;V)^\Omega$ with
$$ \| \Sample_\Omega^{-1} f \|_{L^p(\A_\Z^d;V)} = \exp( O_c(d) ) \| f \|_{\ell^p(\Z^d;V)} $$
for all $f \in \ell^p(\Z^d;V)^{\pi(\Omega)}$, or equivalently
$$ \| \Sample F \|_{\ell^p(\Z^d;V)} = \exp( O_c(d) ) \| F \|_{L^p(\A_\Z^d;V)} $$
for all $F \in L^p(\A_\Z^d;V)^{\Omega}$.
\end{proposition}

\begin{proof}  See \cite[Theorem 4.18]{KMT}, after generalizing from $d=1$ to general $d$ and carefully tracking the dependence on constants.  The result also extends to $0 < p < 1$ (after allowing the implied constants to depend on $p$ as well as $c$), but we will only need the $p \geq 1$ case here.
\end{proof}

Proposition \ref{qss} can be used to partially explain the sampling principle, Proposition \ref{msw}.  First observe that if $[-\eps,\eps]^d \times \Sigma$ is a non-aliasing set then we have the identity
\begin{equation}\label{sam}
 \Op_{m;\Sigma} \Sample f = \Sample \Op_{m \otimes 1_\Sigma} f 
\end{equation}
for any $m \in C^\infty_c(\R^d)$ supported on $[-\eps,\eps]^d$ and any $f \in \Schwartz(\A_\Z^d)$, where the tensor product $\otimes$ is defined in Section \ref{notation-sec}; see \cite[Lemma 4.12]{KMT} (extended to general dimension $d$ in the obvious fashion).   Now suppose that $Q \geq 1$ is an integer and $m \in C^\infty_c(\R^d)$ is supported in $[-\frac{c}{Q},\frac{c}{Q}]^d$ for some $0 < c < \frac{1}{2}$.  Then we may use \eqref{sam} and basic Fourier-analytic manipulations to factorize
\begin{align*}
 \Op_{m;\T^d[Q]} f &= \Op_{m;\T^d[Q]} \Sample \Sample_\Omega^{-1} \Op_{\varphi;\T^d[Q]} f \\
&= \Sample \Op_{m \otimes 1_{\T^d[Q]}} \Sample_\Omega^{-1} \Op_{\varphi;\T^d[Q]} f\\
&= \Sample \Op_{m \otimes 1} \Sample_\Omega^{-1} \Op_{\varphi;\T^d[Q]} f
\end{align*}
where $\varphi \in C^\infty_c(\R^d)$ is a function of the form
$$ \varphi(\xi_1,\dots,\xi_d) \coloneqq \prod_{j=1}^d \varphi_0( Q \xi_j),$$
$\varphi_0 \in \C^\infty_c(\R)$ is supported on $[-c',c']$ for some $c < c' < \frac{1}{2}$ that equals $1$ on $[-c,c]$, and $\Omega \coloneqq [-\frac{c'}{Q},\frac{c'}{Q}]^d \times \T^d[Q]$.  Here $1$ denotes the constant unit function on $(\Q/\Z)^d$, thus $m \otimes 1 \colon \R^d \times (\Q/\Z)^d \to \C$ denotes the function
$$ (m \otimes 1)(\theta,\alpha) \coloneqq m(\theta).$$
From Proposition \ref{msw} applied to $\varphi$ we have
$$ \| \Op_{\varphi;\T^d[Q]} \|_{B(\ell^p(\Z^d;H))} \leq O_{c,c'}(1)^d$$
while from working on each fibre $\R^d \times \{y\}$ of $\A_\Z^d = \R^d \times \hat \Z^d$ separately and using the Marcinkiewicz--Zygmund theorem (Theorem \ref{mz}), we have
$$ \|\Op_{m \otimes 1}\|_{B(L^p(\A_\Z^d;H))}
= \|\Op_{m}\|_{B(L^p(\R^d;H))} = \|\Op_{m}\|_{B(L^p(\R^d))} $$
and hence from Proposition \ref{qss} we have
$$ \| \Op_{m;\Sigma} \|_{B(\ell^p(\Z^d;H))} \leq O_{c,c'}(1)^d \|\Op_{m}\|_{B(L^p(\R^d))} $$
which recovers a slightly weaker version of Proposition \ref{msw}; in fact with a bit more effort (applying a smooth partition of unity to $m$ followed by the triangle inequality) one can in fact recover the full strength of Proposition \ref{msw}.  Admittedly, there is some circularity here since Proposition \ref{msw} was used in the proof, but only for the bump function $\varphi$ and not for arbitrary multipliers $m$.

It turns out that there is a similar phenomenon for major arcs:

\begin{theorem}[Major arc sampling]\label{main-sampling}  Let $(d,k,S,\eps)$ be a major arc parameter set, which is $(r,c)$ good for some $r \geq 1$ and $0 < c < 1$.  Set $\Omega \coloneqq [-\eps,\eps]^d \times \Sigma_{\leq k}$.  Then for any finite-dimensional Hilbert space $H$ and $(2r)' \leq p \leq 2r$, the interpolation operator $\Sample_\Omega^{-1}$ extends to a bounded invertible linear operator from $\ell^p(\Z^d;H)^{\pi(\Omega)}$ to $L^p(\Z^d;H)^\Omega$, with
$$ \| \Sample_\Omega^{-1}f \|_{L^p(\A_\Z^d;H)} = \exp\left( O_c(d) + O( k \Log(r \Log k) ) \right) \| f \|_{\ell^p(\Z^d;H)} $$
for all $f \in \ell^p(\Z^d;H)^{\pi(\Omega)}$, or equivalently
$$ \| \Sample F \|_{\ell^p(\Z^d;H)} = \exp\left( O_c(d) + O( k \Log(r \Log k) )\right ) \| F \|_{L^p(\A_\Z^d;H)} $$
for all $F \in L^p(\A_\Z^d;H)^{\Omega}$.
\end{theorem}

We prove this theorem in Section \ref{proofs-sec}, by reusing the machinery used to establish Theorem \ref{main}.  As a consequence of this sampling theorem, we can obtain a more general ``adelic'' version of the Ionescu--Wainger multiplier theorem, in which one transfers a multiplier on $\A_\Z^d = \R^d \times \hat \Z^d$ (rather than on $\R^d$) to the lattice $\Z^d$, or equivalently one allows the use of a different multiplier on each major arc.  More precisely, given $m \in L^\infty(\R^d \times (\Q/\Z)^d)$ and a finite set $\Sigma$, define the multiplier $\Op_{m;\Sigma} \colon \ell^2(\Z^d) \to \ell^2(\Z^d)$ by the formula
$$ \Op_{m;\Sigma} \coloneqq \sum_{\alpha \in \Sigma} \Op_{m(\cdot,\alpha); \alpha}$$
or equivalently
$$ \Op_{m;\Sigma} f(n) = \sum_{\alpha \in \Sigma} \int_{\R^d} m(\theta,\alpha) e(- n \cdot (\alpha + \theta) ) \F_{\Z^d} f(\alpha + \theta)\ d\theta.$$
Note that the previous definition \eqref{msig} corresponds to the special case in which the adelic symbol $m(\theta,\alpha)$ does not depend on the arithmetic component $\alpha$.

\begin{theorem}[Ionescu--Wainger multiplier theorem, adelic form]\label{main-adelic}  Let $(d,k,S,\eps)$ be a major arc parameter set, and let $H$ be a finite-dimensional Hilbert space.  Let $m \in \Schwartz(\A_\Z^d)$ be supported on $[-\eps,\eps]^d \times (\Q/\Z)^d$.  Then if $(d,k,S,\eps)$ is $(r,c)$-good for some $r \geq 1$ and $0 < c < 1$, one has
$$
\| \Op_{m;\Sigma_{\leq k}} \|_{B(\ell^{p}(\Z^d;H))}
\leq O_c(1)^d O(r \Log k)^{O(k)} \|\Op_m\|_{B(L^{p}(\A_\Z^d))}$$
for any $(2r)' \leq p \leq 2r$.
\end{theorem}

In principle this theorem converts the analysis of linear Fourier multipliers on major arcs to that of linear Fourier multipliers on the adelic space $\A_\Z^d$, which in principle is a simpler setting due to the product structure on $\A_\Z^d = \R^d \times \hat \Z^d$.  We remark that the method of proof also extends to bilinear or multilinear Fourier multipliers (as long as all exponents $p$ involved lie strictly between $1$ and $\infty$), but we do not have applications in mind for this extension\footnote{For instance, the bilinear estimates considered in \cite{KMT} typically involve the endpoint space $\ell^1$ (or even $\ell^p$ for some $p<1$), and also take values in variational norm spaces rather than Hilbert spaces, so would not be able to be directly treated by a bilinear variant of Theorem \ref{main-adelic}.} and so we leave it to the interested reader.  As this theorem no longer requires $p$ to be an even integer or its dual, it affirmatively answers the question posed after \cite[Theorem 2.1]{MSZ3}.

\subsection{Acknowledgments}

The author was partially supported by NSF grant DMS-1764034 and by a Simons Investigator Award.  The author also thanks Mariusz Mirek for helpful comments and corrections, Lillian Pierce for pointing out the reference \cite{K}, and commenters on the author's blog for further corrections.  We are particularly indebted to the anonymous referee for a very careful reading of the manuscript with many useful suggestions.

\subsection{Notation}\label{notation-sec}

Random variables will be denoted in boldface, and deterministic quantities in non-boldface.  We use $\N =\{0,1,\dots\}$ to denote the natural numbers, and $\Z_+ = \{1,2,\dots\}$ to denote the positive integers.

We use $X = O(Y)$ to denote an estimate of the form $|X| \leq CY$ for some constant $C$. We write $X \sim Y$ if $X = O(Y)$ and $Y = O(X)$. If one needs the constant $C$ to depend on parameters, we indicate this by subscripts, for instance $X \leq O_c(1)^d Y$ denotes the bound $X \leq C_c^d Y$ for some $C_c$ depending only on $c$. 

If $(X,\mu)$ is a measure space, $V = (V,\| \|)$ is a finite-dimensional normed vector space, and $1 \leq p \leq \infty$, we define $L^p(X;V)$ to denote the space of measurable functions $f: X \to V$ whose norm $\|f\|_{L^p(X;V)} \coloneqq (\int_X \|f(x)\|_V^p\ d\mu(x))^{1/p}$ is finite, up to almost everywhere equivalence, with the usual modifications at $p=\infty$.  We write $L^p(X)$ for $L^p(X;\C)$, and when $\mu$ is counting measure we write $\ell^p$ for $L^p$.  For any $1 \leq p \leq \infty$, we define the dual exponent $1 \leq p' \leq \infty$ by $1/p + 1/p' = 1$.

All Hilbert spaces will be over the complex numbers.  Given a bounded linear operator $T \colon V \to W$ between (quasi-)normed vector spaces $V, W$, we use $\|T\|_{B(V \to W)}$ to denote its operator norm; if $V=W$, we abbreviate $B(V \to V)$ as $B(V)$. We recall

\begin{theorem}[Marcinkiewicz--Zygmund theorem]\label{mz}\cite{mz}  Let $X,Y$ be measure spaces, let $0 < p < \infty$, and let $T: L^p(X) \to L^p(Y)$ be a linear operator.  Then for any finite-dimensional Hilbert space $H$, one has
$$ \| T \|_{B(L^p(X;H) \to L^p(Y;H))} \leq \| T \|_{B(L^p(X) \to L^p(Y))}.$$
\end{theorem}

\begin{proof}  We normalize $\| T \|_{B(L^p(X) \to L^p(Y))}=1$.
Taking orthonormal bases, it suffices to show that
$$ \int_Y \left(|\sum_{i=1}^n Tf_i|^2\right)^{p/2} \leq \int_X \left(|\sum_{i=1}^n f_i|^2\right)^{p/2}$$
for any $f_1,\dots,f_n \in L^p(X)$.  If we let ${\bf g}_1,\dots,{\bf g}_n$ be independent complex gaussian variables of mean zero and variance $1$, we have from hypothesis that
$$ \int_Y |\sum_{i=1}^n {\bf g}_i Tf_i|^p \leq \int_X |\sum_{i=1}^n {\bf g}_i f_i|^p.$$
Taking expectations of both sides and noting that the sum of independent gaussians is again a gaussian, we conclude that
$$ C_p \int_Y \left(|\sum_{i=1}^n Tf_i|^2\right)^{p/2} \leq C_p \int_X \left(|\sum_{i=1}^n f_i|^2\right)^{p/2}$$
where $C_p \coloneqq \E |{\bf g}|^p$ with ${\bf g}$ a complex gaussian of mean zero and variance $1$. Since $0 < C_p < \infty$, the claim follows.
\end{proof}

If $E$ is a finite set, we use $|E|$ to denote its cardinality.
If $E,F$ are subsets of an additive group $\G = (\G,+)$ (such as the torus $\T^d)$, we write $E+F \coloneqq \{ \xi+\eta: \xi \in E, \eta \in F \}$ for their sumset.  If $\xi \in \G$, we write $\xi + E = E + \xi = E + \{\xi\}$ for the translate of $\xi$ by $E$.  These sumset notions also extend in the obvious fashion to the setting in which one of the summands lies in $\G$ and the other lies in a quotient $\G/\HH$ (for instance, if one lies in $\R^d$ and the other in $\T^d$).  We use $1_E$ to denote the indicator function of a set $E$, and $1_S$ the indicator of a statement $S$, thus for instance $1_E(x) = 1_{x \in E}$ is equal to $1$ when $x \in E$, and equal to $0$ otherwise.

We will need the following combinatorial concepts:

\begin{definition}[Nonces and sunflowers]\label{nonceflower}  Let $A_1,\dots,A_n$ be a collection of sets.
\begin{itemize}
\item[(i)]  A \emph{nonce} of the collection $A_1,\dots,A_n$ is an element $s$ that belongs to exactly one of the $A_i$.  A collection of sets is \emph{nonce-free} if there does not exist a nonce for this collection.  
\item[(ii)]  The collection $A_1,\dots,A_n$ is a \emph{sunflower} if there is a set $A_0$ contained in each of the $A_1,\dots,A_n$ (the \emph{core} of the sunflower) such that the \emph{petals} $A_1 \backslash A_0,\dots,A_n \backslash A_0$ are all disjoint.
\end{itemize}
\end{definition}

Thus for instance the sets $\{1,2\}, \{1,3\}, \{2,4\}$ contain $3$ and $4$ as nonces, whereas the collection $\{1,2\}, \{1,3\}, \{2,3\}$ is nonce-free.  Meanwhile, the collection $\{1,2\}, \{1,3\}, \{1,4\}$ is a sunflower with core $\{1\}$ and petals $\{2\}, \{3\}, \{4\}$.  The property of having a nonce is also referred to as the \emph{uniqueness property} in \cite{IW}, \cite{MSZ3}, \cite{pierce}.

If $f \colon X \to \C$ and $g \colon Y \to \C$ are functions, we define the tensor product $f \otimes g \colon X \times Y \to C$ by the formula
$$ (f \otimes g)(x,y) \coloneqq f(x) g(y).$$

If $H$ is a finite-dimensional Hilbert space and $S$ is a finite set, we use $H^S$ for the space of tuples $(u_s)_{s \in S}$ with $u_s \in H$ with inner product
$$ \langle (u_s)_{s \in S}, (v_s)_{s \in S}\rangle = \sum_{s \in S} \langle u_s, v_s \rangle.$$
For any natural number $k$, we use $H^{\otimes k}$ to denote the $k$-fold tensor product of $H$ with itself, spanned by vectors $u_1 \otimes \dots \otimes u_k$, $u_1,\dots,u_k \in H$ with
$$ \langle u_1 \otimes \dots \otimes u_k, v_1 \otimes \dots \otimes v_k \rangle = \prod_{i=1}^k \langle u_i, v_i \rangle.$$

We use $X \uplus Y$ to denote the disjoint union of $X$ and $Y$, thus $X \uplus Y$ is equal to $X \cup Y$ when $X,Y$ are disjoint and undefined otherwise.

\section{Superorthogonality}\label{super-sec}

The (upper) Khintchine inequality asserts that
$$ \left(\E |\sum_{i=1}^n \bm{\epsilon}_i z_i|^p\right)^{1/p} \leq O(p^{1/2}) \left(\sum_{i=1}^n |z_i|^2\right)^{1/2}$$
for any $1 \leq p < \infty$ and complex numbers $z_1,\dots,z_n$, where $\bm{\epsilon}_1,\dots,\bm{\epsilon}_n$ are independent random signs in $\{-1,+1\}$ of mean zero.  In the case where $p=2r$ is an even integer, this inequality can be proven by direct combinatorial expansion of the left-hand side.  As laid out recently in \cite{pierce}, this latter argument can be abstracted to more general ``Type II superorthogonal systems''.  We give the relevant definitions (as well as an extension to hypersystems) as follows.

\begin{definition}[Type II superorthogonality]\label{superortho-def}  Let $S$ be a finite set, let $X = (X,\mu)$ be a measure space, let $r$ be a positive integer, and let $H$ be a Hilbert space.
\begin{itemize}
\item[(i)] A collection $(f_s)_{s \in S}$ of functions $f_s \in L^{2r}(X;H)$ indexed by $S$ is said to be a \emph{Type II $2r$-superorthogonal system}  if one has
\begin{equation}\label{2r-orthog}
\int_X \prod_{j=1}^r \langle f_{s_j}, f_{s_{r+j}} \rangle_H\ d\mu = 0
\end{equation}
whenever $s_1,\dots,s_{2r} \in S$ is such that the singleton sets $\{s_1\},\dots,\{s_{2r}\}$ contain a nonce (as defined in Definition \ref{nonceflower}).
\item[(ii)]  A collection $(f_A)_{A \in {\mathcal A}}$ of functions $f_A \in L^{2r}(X;H)$ indexed by some family ${\mathcal A}$ of subsets of a set $S$ is said to be a \emph{Type II $2r$-superorthogonal hypersystem} if one has
\begin{equation}\label{2r-orthog-hyper}
\int_X \prod_{j=1}^r \langle f_{A_j}, f_{A_{r+j}} \rangle_H\ d\mu = 0
\end{equation}
whenever $A_1,\dots,A_{2r} \in {\mathcal A}$ is such that the sets $A_1,\dots,A_{2r}$ contain a nonce.
\end{itemize}
\end{definition}

Note that any $2r$-superorthogonal system $(f_s)_{s \in S}$ can also be viewed as a $2r$-superorthogonal hypersystem $(f_A)_{A \in \binom{S}{1}}$ by identifying each element $s \in S$ with the associated singleton $\{s\} \in \binom{S}{1}$. The nomenclature ``Type II'' is due to Pierce \cite{pierce}; there is also a stronger notion of Type I superorthogonality and a weaker notion of Type III superorthogonality discussed in that paper, but we will not need these notions here.

Several examples of superorthogonal systems are given in \cite{pierce}.  Our primary concern will come from functions supported on major arcs, but we can give another representative example of a superorthogonal hypersystem here:

\begin{example}[Polynomials of random variables]\label{poly-rand}
Let $k,R$ be positive integers.  Let $({\bf X}_s)_{s \in S}$ be $R$-wise independent random variables (thus ${\bf X}_{s_1},\dots,{\bf X}_{s_r}$ are jointly independent for any $r \leq R$ and distinct $s_1,\dots,s_r \in S$), and for each $A \in \binom{S}{\leq k}$ let ${\bf f}_A$ be a complex random variable of the form ${\bf f}_A = \sum_{j \in J_A} c_{A,j} \prod_{s \in A} f_{A,s,j}({\bf X}_s)$, where $J_A$ is a finite set, $c_{A,j}$ are complex coefficients, and each $f_{A,s,j}({\bf X}_s)$ is a function of ${\bf X}_s$ of mean zero; thus for instance ${\bf f}_\emptyset$ is a constant.  Then for any $1 \leq r \leq R/2k$, $({\bf f}_A)_{A \in \binom{S}{\leq k}}$ is a $2r$-superorthogonal hypersystem over the ambient sample space of the random variables.  Indeed, if $A_1,\dots,A_{2r}$ contains a nonce $s$, then the expression in \eqref{2r-orthog-hyper} expands into a sum of finitely many terms, each of which consists of the expectation of a product of an expression of the form $f_{A,s,j}({\bf X}_s)$, times at most $2kr-1 \leq R-1$ other expressions depending on other random variables than ${\bf X}_s$, and each of these terms vanishes by the $2R$-wise independent nature of the ${\bf X}_s$.  If the ${\bf X}_s$ are scalar random variables, then any polynomial of degree at most $k$ in the ${\bf X}_s$ can be expressed in the form $\sum_{A \in \binom{S}{\leq k}} {\bf f}_A$ for some hypersystem $({\bf f}_A)_{A \in \binom{S}{\leq k}}$ as above by removing the expectation $\E {\bf X}_s^a$ from every monomial ${\bf X}_s^a$ appearing in the polynomial and regrouping terms.
\end{example}

We now give the Khintchine inequalities for superorthogonal systems and hypersystems.

\begin{theorem}[Superorthogonal Khintchine inequality]\label{sys-ortho}  Let $k,r \in \Z_+$, let $X = (X,\mu)$ be a measure space, $S$ a finite set, and $H$ a finite-dimensional Hilbert space.
\begin{itemize}
\item[(i)]  (Khintchine for superorthogonal systems) If $(f_s)_{s \in S}$ is a Type II $2r$-superorthogonal system in $L^{2r}(X;H)$ indexed by $S$, then
$$ 
\| \sum_{s \in S} f_s \|_{L^{2r}(X;H)} \leq O(r)^{1/2} \left\| (f_s)_{s \in S} \right \|_{L^{2r}(X; H^S)}.$$
\item[(ii)]  (Khintchine for superorthogonal hypersystems) If $(f_A)_{A \in \binom{S}{\leq k}}$ is a Type II $2r$-superorthogonal hypersystem in $L^{2r}(X;H)$ then
$$ 
\| \sum_{A \in \binom{S}{\leq k}} f_A \|_{L^{2r}(X;H)} \leq O(r)^{k/2} \left\| (f_A)_{A \in \binom{S}{\leq k}} \right \|_{L^{2r}(X; H^{[\leq k]})}$$
where we adopt the notation
$$ H^{[\leq k]} \coloneqq H^{\binom{S}{\leq k}}.$$
\end{itemize}
\end{theorem}

Part (i) is standard (see e.g., \cite[\S 3.1]{pierce}).  Part (ii) (without any losses of $\Log |S|$ factors) appears to new; with logarithmic losses one can obtain a result of this type from \cite[Lemma 5.1]{MST}, an iteration of (i), and the triangle inequality.  In more pedestrian notation, we may write
$$ \left\| (f_s)_{s \in S} \right \|_{L^{2r}(X; H^S)} = \left\| \left( \sum_{s \in S} \|f_s\|_H^2\right)^{1/2} \right \|_{L^{2r}(X)}$$
and similarly
$$ \left\| (f_A)_{A \in \binom{S}{\leq k}} \right \|_{L^{2r}(X; H^{[\leq k]})} = \left\| \left( \sum_{A \in \binom{S}{\leq k}} \|f_A\|_H^2\right)^{1/2} \right \|_{L^{2r}(X)}.$$

\begin{proof}  We begin with (i).  We may index $S = \{1,\dots,n\}$.  The desired estimate may be rewritten as
$$ \int_X \|\sum_{s=1}^n f_s\|_H^{2r}\ d\mu \leq O(r)^r \int_X (\sum_{s=1}^n \|f_s\|_H^2)^{r}\ d\mu.$$
From the binomial theorem, the Cauchy-Schwarz inequality, and the triangle inequality, for any $u,v \in H$ we have
$$ \|u+v\|_H^{2r} = \|u\|_H^{2r} + 2r \mathrm{Re} \langle v, u \rangle \|u\|_H^{2r-2} + O\left( \sum_{j=2}^{2r} \binom{2r}{j} \|v\|_H^j \|u\|_H^{2r-j} \right).$$
Observe for any odd $2j+1$ between $1$ and $2r$ that
$$ \binom{2r}{2j+1} \sim \binom{2r}{2j}^{1/2} \binom{2r}{2j+2}^{1/2}$$
(since $k! \sim ((k-1)! (k+1)!)^{1/2}$ for any $k \in \Z_+$), and hence by Young's inequality $ab \leq \frac{1}{2} a^2 + \frac{1}{2} b^2$ we have
$$
\binom{2r}{2j+1} \|v\|_H^{2j+1} \|u\|_H^{2r-2j-1} \lesssim
\binom{2r}{2j} \|v\|_H^{2j} \|u\|_H^{2r-2j} + \binom{2r}{2j+2} \|v\|_H^{2j+2} \|u\|_H^{2r-2j-2}.$$
As a consequence, we may restrict the $j$ summation here to even integers, thus
$$ \|u+v\|_H^{2r} = \|u\|_H^{2r} + 2r \mathrm{Re} \langle v, u \rangle \|u\|_H^{2r-2}  + O\left( \sum_{j=1}^{r} \binom{2r}{2j} \|v\|_H^{2j} \|u\|_H^{2r-2j} \right)$$
and in particular
$$ \|u+v\|_H^{2r} \leq 2r \mathrm{Re} \langle v, u \rangle \|u\|_H^{2r-2}  + \sum_{j=0}^{r} C^{1_{j \geq 1}} \binom{2r}{2j} \|v\|_H^{2j} \|u\|_H^{2r-2j}$$
for some absolute constant $C > 1$.  As a special case, for any $v_1,\dots,v_n \in H$, one has
$$ \left\|\sum_{s=1}^n v_s\right\|_H^{2r} \leq \mathrm{Re} Z + \sum_{j=0}^r C^{1_{j \geq 1}} \binom{2r}{2j} \|v_1\|_H^{2j} \left\|\sum_{s=2}^n v_s\right\|_H^{2r-2j}$$
where $Z$ is a linear combination of expressions of the form $\prod_{j=1}^r \langle v_{s_j}, v_{s_{r+j}} \rangle$ where $\{s_1\},\dots,\{s_{2r}\}$ contains a nonce.  Iterating this identity $n$ times, we conclude that
$$\left \|\sum_{s=1}^n v_s\right\|_H^{2r} \leq \mathrm{Re} Z' + \sum^* C^{1_{j_1 \geq 1}+\dots+1_{j_n \geq 1}} \binom{2r}{2j_1,\dots,2j_n} \|v_1\|_H^{2j_1}\dots \|v_n\|_H^{2j_n}$$
where $Z'$ is also a linear combination of expressions of the form $\prod_{j=1}^r \langle v_{s_j}, v_{s_{r+j}} \rangle$ with $\{s_1\},\dots,\{s_{2r}\}$ containing a nonce, and $\sum^*$ denotes a sum over tuples $(j_1,\dots,j_n) \in \N^n$ with $j_1+\dots+j_n=r$.  
Applying this with $v_i \coloneqq f_i(x)$, integrating in $X$, and using the Type II $2r$-superorthogonality hypothesis \eqref{2r-orthog} as well as the bound $C^{1_{j_1 \geq 1}+\dots+1_{j_n \geq 1}}  \leq C^r$, we conclude that
$$ \int_X \|\sum_{s=1}^n f_s\|_H^{2r}\ d\mu \leq \sum^* C^r \binom{2r}{2j_1,\dots,2j_n} \int_X \|f_1\|_H^{2j_1}\dots \|f_n\|_H^{2j_n}\ d\mu.$$
On the other hand, we have
$$ \int_X (\sum_{s=1}^n \|f_s\|_H^2)^{r} = \sum^* \binom{r}{j_1,\dots,j_n} \int_X \|f_1\|_H^{2j_1}\dots \|f_n\|_H^{2j_n}\ d\mu.$$
To finish the proof it will suffice to establish the inequality
$$ \binom{2r}{2j_1,\dots,2j_n} \leq (2r)^r  \binom{r}{j_1,\dots,j_n}$$
for any $j_1,\dots,j_n \geq 0$ summing to $r$.  But this follows from the combinatorial observation that given a partition of $\{1,\dots,2r\}$ into $n$ classes of cardinality $2j_1,\dots,2j_n$ respectively, one can remove $j_1$ elements from the first class, then $j_2$ elements from the second class, and so forth until one is left with a partition of $r$ elements of $\{1,\dots,2r\}$ into $n$ classes of cardinality $j_1,\dots,j_n$.  There are at most $(2r)^r$ ways to remove these elements in the order indicated, and $\binom{r}{j_1,\dots,j_n}$ ways to partition the remaining elements, with the original partition being recoverable from this data.  This gives (i).

Now we prove (ii).  By the triangle inequality we have
$$ 
\| \sum_{A \in \binom{S}{\leq k}} f_A \|_{L^{2r}(X;H)} \leq \sum_{k' \leq k} \| \sum_{A \in \binom{S}{k}} f_A \|_{L^{2r}(X;H)}$$
and also $\sum_{k' \leq k} O(r)^{k'/2} = O(r)^{k/2}$, so it suffices to establish the inequality
$$ 
\| \sum_{A \in \binom{S}{k}} f_A \|_{L^{2r}(X;H)} \leq O(r)^{k/2} \left\| (f_A)_{A \in \binom{S}{k}} \right \|_{L^{2r}(X; H^{\binom{S}{k}})}$$
for every $k \geq 0$.  The case $k=0$ is trivial, so suppose $k \geq 1$.  We apply the probablistic decoupling method (cf., \cite{pena}), which can be viewed as a substitute for the random partitioning lemma in \cite[Lemma 5.1]{MST} that avoids logarithmic losses.  We form a random partition $S = {\bf S}_1 \uplus \dots \uplus {\bf S}_k$ by setting ${\bf S}_i \coloneqq \{ s \in S: {\bf i}_s = i \}$, where ${\bf i}_s, s \in S$ are independent random variables drawn uniformly at random from $\{1,\dots,k\}$.  Observe that if $A \in \binom{S}{k}$, then $A$ takes the form $A = \{s_1,\dots,s_k\}$ with $s_i \in {\bf S}_i$ for $i=1,\dots,k$ with probability precisely $\frac{k!}{k^k}$ (this is the probability that the tuple $(i_s)_{s \in A}$ forms a permutation of $\{1,\dots,k\}$).  Thus we have
$$ \sum_{A \in \binom{S}{k}} f_A  = \frac{k^k}{k!} \E \sum_{s_1 \in {\bf S}_1,\dots,s_k \in {\bf S}_k} f_{\{s_1,\dots,s_k\}}$$
and hence by the triangle inequality
$$ 
\| \sum_{A \in \binom{S}{k}} f_A \|_{L^{2r}(X;H)} \leq \frac{k^k}{k!} \E \left\| \sum_{s_1 \in {\bf S}_1,\dots,s_k \in {\bf S}_k} f_{\{s_1,\dots,s_k\}} \right\|_{L^{2r}(X;H)}.$$
Using the Taylor expansion
$$ e^k = \frac{k^0}{0!} + \frac{k^1}{1!} + \dots + \frac{k^k}{k!} + \dots \geq \frac{k^k}{k!}$$
it will thus suffice to establish the deterministic inequality
$$
\left\| \sum_{s_1 \in S_1,\dots,s_k \in S_k} f_{\{s_1,\dots,s_k\}} \right\|_{L^{2r}(X;H)}
\leq O(r)^{k/2}
\left\| (f_{\{s_1,\dots,s_k\}})_{s_1 \in S_1,\dots,s_k \in S_k} \right\|_{L^{2r}(X; H^{S_1 \times \dots \times S_k})}$$
whenever $S = S_1 \uplus \dots \uplus S_k$ is a partition of $S$.  By induction, it suffices to establish the bound
\begin{align*}
&\left\| \left(\sum_{s_i \in S_i,\dots,s_k \in S_k} f_{\{s_1,\dots,s_k\}}\right)_{s_1 \in S_1,\dots,s_{i-1} \in S_{i-1}} \right\|_{L^{2r}(X;H^{S_1 \times \dots \times S_{i-1}})}\\
&\quad \leq O(r)^{1/2}
\left\| \left(\sum_{s_{i+1} \in S_{i+1},\dots,s_k \in S_k} f_{\{s_1,\dots,s_k\}}\right)_{s_1 \in S_1,\dots,s_i \in S_i} \right\|_{L^{2r}(X;H^{S_1 \times \dots \times S_i})}
\end{align*}
for all $1 \leq i \leq k$.  But this follows by applying part (i) to the Hilbert space $H^{S_1 \times \dots \times S_{i-1}}$ and the functions
$$ \vec f_{s_i} \coloneqq \left( \sum_{s_{i+1} \in S_{i+1},\dots,s_k \in S_k} f_{\{s_1,\dots,s_k\}} \right)_{s_1 \in S_1,\dots,s_{i-1} \in S_{i-1}}$$
for $s_i \in S_i$, as it is not difficult to show that these functions form a $2r$-superorthogonal system in $L^{2r}(X; H^{S_1 \times \dots \times S_{i-1}})$.
\end{proof}

As a sample application of Theorem \ref{sys-ortho}, we can specialize to the situation in Example \ref{poly-rand} to conclude

\begin{corollary}[Hoeffding-type inequality]\label{hoeff}  Let the notation and hypotheses be as in Example \ref{poly-rand}. If we have the bound
\begin{equation}\label{as2}
 \sum_{A \in \binom{S}{\leq k}: A \neq \emptyset} |{\bf f}_A|^2 \leq \sigma^2
\end{equation}
almost surely for some $\sigma>0$, then one has
 \begin{equation}\label{stim}
\P\left( \left| \sum_{A \in \binom{S}{\leq k}} {\bf f}_A - {\bf f}_\emptyset\right| \geq \lambda \sigma \right) \leq O(1)^k \left(\exp(-ck \lambda^{2/k}) + \exp( -c R)\right)
\end{equation}
for all $\lambda > 0$ and some absolute constant $c>0$.
\end{corollary}

See \cite{bp}, \cite{bbb} for some previous Hoeffding-type inequalities for sums of $R$-wise independent random variables.

\begin{proof}  We may normalize ${\bf f}_\emptyset=0$.  By shrinking $\lambda$ if necessary we can also assume that $\lambda \leq (R/k)^{k/2}$ (otherwise the first term on the right-hand side is dominated by the second).  We can also assume that $\lambda \geq C^k$ and $R \geq Ck$ for a large constant $C$, as the bound is trivial otherwise.
Let $1 \leq r \leq R/2k$ be an integer to be chosen later.  By Markov's inequality one has
$$\P\left( \left| \sum_{A \in \binom{S}{\leq k}} {\bf f}_A \right| \geq \lambda \sigma \right) \leq \lambda^{-2r} \sigma^{-2r}
\E \left| \sum_{A \in \binom{S}{\leq k}} {\bf f}_A \right|^{2r}.$$
Applying Theorem \ref{sys-ortho}(ii) to the $2r$-superorthogonal hypersystem $({\bf f}_A)_{A \in \binom{S}{\leq k}}$, we obtain
$$ \E \left| \sum_{A \in \binom{S}{\leq k}} {\bf f}_A \right|^{2r} \leq O(r)^{kr} \E \left(\sum_{A \in \binom{S}{\leq k}} | {\bf f}_A |^2\right)^{r}.$$
Combining this with the preceding inequality and \eqref{as2}, we conclude that
$$\P\left( \left| \sum_{A \in \binom{S}{\leq k}} {\bf f}_A \right| \geq \lambda \sigma \right) \leq (O(r) / \lambda^{2/k})^{kr}.$$

If we set $r \coloneqq \lfloor c \lambda^{2/k} \rfloor$ for a sufficiently small absolute constant $c>0$, we obtain the claim.
\end{proof}

We now discuss the sharpness of the estimates in Corollary \ref{hoeff}.  The first example below shows that the first term on the right-hand side of \eqref{stim} is reasonably sharp; the second example shows the second term in \eqref{stim} only has a small amount of room for improvement.

\begin{example}  Let $n,k$ be positive integers, and let ${\bf X}_{i,j}$ for $i=1,\dots,k$ and $j=1,\dots,n$ be independent Bernoulli random variables taking values in $\{-1,+1\}$ of mean zero.  Then the random variable $\prod_{i=1}^k \sum_{j=1}^n {\bf X}_{i,j}$ can be expanded in the form $\sum_A {\bf f}_A$ where ${\bf f}_A = {\bf X}_{1,j_1} \dots {\bf X}_{k,j_k}$ when $A$ is of the form $\{(1,j_1),\dots,(k,j_k)\}$ and ${\bf f}_A=0$ otherwise.  One then easily verifies that \eqref{as2} holds with $\sigma=n^{k/2}$, and that
$$ \P\left( \left| \sum_{A \in \binom{S}{\leq k}} {\bf f}_A \right| \geq \lambda \sigma \right) = 2^{-nk}$$
when $\lambda = n^{k/2}$ and $S = \{1,\dots,k\} \times \{1,\dots,n\}$.  Here one can take $R$ to be arbitrary. This shows that the first-term on the right-hand side of \eqref{stim} cannot be improved except possibly for the $O(1)^k$ factor or in the explicit value of $c$. Modifications of this example can also be used to illustrate the sharpness of Theorems \ref{sys-ortho}; we leave the details to the interested reader.
\end{example}

\begin{example}  Let $R$ be a natural number, let $p$ be a prime greater than $R$, and let ${\bf P}$ be a random polynomial of degree at most $R-1$ with coefficients in $\Z/p\Z$, drawn uniformly among all such polynomials.  Then the random variables ${\bf P}(i)$ for $i=1,\dots,p$ are $R$-wise independent, since the Lagrange interpolation formula shows that for any distinct $i_1,\dots,i_R$, the map from polynomials ${\bf P}$ to evaluations $({\bf P}(i_1),\dots,{\bf P}(i_{R}))$ is a bijection.  We then have
$$ \P\left( \left|\sum_{i=1}^p (1_{{\bf P}(i)=0} - 1/p)\right| = p-1 \right) = p^{-R}$$
Comparing this with \eqref{stim} with $\sigma^2 = p$, $k=1$, and $\lambda = \frac{p-1}{\sqrt{p}}$, and taking $p$ comparable to $C R \log R$, we see that the $\exp(-cR)$ term in \eqref{stim} cannot be improved to more than $\exp(-C R \log R)$ for some constant $C$. One can construct similar examples for higher values of $k$ by considering the random variable $\prod_{j=1}^k \sum_{i=1}^p (1_{{\bf P}_j(i)=0} - 1/p)$ where ${\bf P}_1,\dots,{\bf P}_k$ are independent copies of ${\bf P}$; we leave the details to the interested reader.
\end{example}

As another application of Theorem \ref{sys-ortho} we give\footnote{We thank Nikolay Tzvetkov for this suggestion.} a (slightly weaker form) of a standard Wiener chaos estimate.

\begin{corollary}[Wiener chaos estimate]  Let $S$ be a finite set, and let $\mathbf{g}_s, s \in S$ be independent real gaussian variables of mean zero and variance one.  Let $k \in \N$, and for each $A \in \binom{S}{\leq k}$ let $c_A$ be a an element of a finite-dimensional Hilbert space $H$.  Then for any $2 \leq p < \infty$, one has
$$ \E \left\| \sum_{A \in \binom{S}{\leq k}} c_A \prod_{s \in A} \mathbf{g}_s \right\|_H^p
\leq O(p)^{kp/2} (\sum_{A \in \binom{S}{\leq k}} \|c_A\|_H^2)^{p/2}.$$
\end{corollary}

Using hypercontractivity inequalities, one can show that the factor $O(p)^{kp/2}$ can be improved to $(p-1)^{kp/2}$; see for instance \cite[Theorem I.22]{simon}. Thus we see that Khintchine type inequalities can be used as a partial substitute for hypercontractivity inequalities in some settings.

\begin{proof}  By interpolation it suffices to establish this bound when $p=2r$ is an even integer.  A direct application of Theorem \ref{sys-ortho} loses an additional factor of $O(p)^{kp/2}$ due to the unboundedness of the gaussian random variables $\mathbf{g}_s$.  To avoid this loss we exploit the central limit theorem.  Let $N$ be a large number, and for $s \in S$ and $i=1,\dots,N$ let $\bm{\epsilon}_{s,i}$ be independent Bernoulli variables taking values in $\{-1,1\}$ with probability $1/2$ of each.  By the central limit theorem, $(\frac{1}{\sqrt{N}} \sum_{i=1}^N \bm{\epsilon}_{s,i})_{s \in S}$ converges in distribution to $(\mathbf{g}_s)_{s \in S}$ as $N \to \infty$.  Thus by Fatou's lemma, it suffices to show that
$$ \E \left\| \sum_{A \in \binom{S}{\leq k}} c_A \prod_{s \in A} (\frac{1}{\sqrt{N}} \sum_{i=1}^N \bm{\epsilon}_{s,i}) \right\|_H^{2r}
\leq O(r)^{kr} (\sum_{A \in \binom{S}{\leq k}} \|c_A\|_H^2)^r$$
uniformly in $N$.  The left-hand side can be expanded as
$$ \E \left\| \sum_{A \in \binom{S}{\leq k}} \sum_{(i_s)_{s \in A} \in \{1,\dots,N\}^A} N^{-|A|/2} c_A \prod_{s \in A} \bm{\epsilon}_{s,i_s} \right\|_H^{2r}.$$
By Example \ref{poly-rand}, the $c_A \sum_{s \in A} \bm{\epsilon}_{s,i_s}$ form a Type II $2r$-orthogonal hypersystem indexed by $S \times \{1,\dots,N\}$, so by Theorem \ref{sys-ortho}(ii) we can bound the above expression by
$$ O(r)^{kr}
\E\left( \sum_{A \in \binom{S}{\leq k}} \sum_{(i_s)_{s \in A} \in \{1,\dots,N\}^A} N^{-|A|} \left\|c_A \prod_{s \in A} \bm{\epsilon}_{s,i_s} \right\|_H^2\right)^r.$$
But as the $\bm{\epsilon}_{s,i_s}$ have magnitude $1$, this simplifies to
$$ O(r)^{kr} \left(\sum_{A \in \binom{S}{\leq k}} \|c_A\|_H^2\right)^r$$
as desired.
\end{proof}

\section{Sunflower bound}\label{sunflower-sec}

If $k,r \in \Z_+$, let $\Sun(k,r)$ denote the smallest natural number with the property that any family of $\Sun(k,r)$ distinct sets of cardinality at most $k$ contains $r$ distinct elements $A_1,\dots,A_r$ that form a sunflower (as defined in Definition \ref{nonceflower}).  The celebrated \emph{Erd\H{o}s-Rado theorem} \cite{erdos} asserts that $\Sun(k,r)$ is finite; in fact Erd\H{o}s and Rado gave the bounds
$$ (r-1)^k \leq \Sun(k,r) \leq (r-1)^k k! + 1.$$
The \emph{sunflower conjecture} asserts in fact that the upper bound can be improved to $\Sun(k,r) \leq O(1)^k r^k$.  This remains open at present; the best bound known currently (in the regime where $k,r$ are both large) is
\begin{equation}\label{rao}
\Sun(k,r) \leq O( r \Log(k) )^k
\end{equation}
for all $k,r \in \Z_+$, due to Bell, Chueluecha, and Warnke \cite{bcw}, who modified an argument of Rao \cite{rao} (which in turn built upon a recent breakthrough of Alweiss, Lovett, Wu, and Zhang \cite{alwz}).

We can give a probabilistic version of the Erd\H{o}s-Rado theorem:

\begin{lemma}[Probabilistic Erd\H{o}s-Rado theorem]\label{per}  Let $k,r \in \Z_+$, let $S$ be a finite set, and let ${\bf A}$ be a random subset of $S$ of cardinality $k$ (i.e., a random element of $\binom{S}{k}$).  (We do not require the distribution of ${\bf A}$ to be uniform.) Let ${\bf A}_1,\dots,{\bf A}_r$ be $r$ independent copies of ${\bf A}$.  Then with probability at least $ (4 \Sun(k,r))^{-r}$, ${\bf A}_1,\dots,{\bf A}_r$ form a sunflower.
\end{lemma}

\begin{proof}  The $r=1$ case is trivial, so we may assume $r>1$, in particular $\Sun(k,r) \geq r \geq 2$.
If there is a set $A \in \binom{S}{k}$ with $\P( {\bf A} = A ) \geq (4\Sun(k,r))^{-1}$, then with probability at least $(4\Sun(k,r))^{-r}$ we have ${\bf A}_1=\dots={\bf A}_r=A$.  Since $A,\dots,A$ is a sunflower, this gives the claim.

Now suppose that $\P( {\bf A} = A ) < (4\Sun(k,r))^{-1}$ for all $A \in \binom{S}{k}$.  We form $2\Sun(k,r)$ independent samples ${\bf A}_1,\dots,{\bf A}_{2\Sun(k,r)}$ of ${\bf A}$.  Consider the event $E$ that these samples only consist of at most $\Sun(k,r)$ distinct sets.  If this event occurs, then there are $m$ distinct sets amongst the samples, each of them occurring with some multiplicities $a_1,\dots,a_m$ summing to $2\Sun(k,r)$.  The number of ways to create a maximal collection ${\bf A}_{i_1},\dots,{\bf A}_{i_m}$ of distinct samples is then $a_1 \dots a_m$, which by the arithmetic mean-geometric mean inequality is bounded by $(2\Sun(k,r)/m)^m$, which is in turn bounded by $e^{2\Sun(k,r)/e}$ using the standard inequality $x^{1/x} \leq e^{1/e}$ for $x \geq 1$ applied to $x = 2\Sun(k,r)/m$.
On the other hand, if we fix these indices $i_1,\dots,i_m$ for some $m \leq \Sun(k,r)$, we see from hypothesis that each of the other samples ${\bf A}_i$ has a probability at most $1/4$ of matching one of these distinct samples.  From the union bound, we conclude that
$$ \P(E) \leq e^{2\Sun(k,r)/e} (1/4)^{\Sun(k,r)} \leq \frac{1}{2}$$
since $\Sun(k,r) \geq 2$.  If we now condition to the complement of $E$, the samples ${\bf A}_1,\dots,{\bf A}_{2\Sun(k,r)}$ necessarily contain a sunflower, by definition of $\Sun(k,r)$.  Undoing the conditioning, we conclude that ${\bf A}_1,\dots,{\bf A}_{2\Sun(k,r)}$ contains a sunflower with probability at least $1/2$.  By symmetry, this means that ${\bf A}_1,\dots,{\bf A}_r$ is a sunflower with probability at least
$$ 2^{-1} \binom{2\Sun(k,r)}{r}^{-1} \geq 2^{-1} (2\Sun(k,r))^{-r},$$
giving the claim.
\end{proof}

In the converse direction, we can find a collection $A_1,\dots,A_{\Sun(k,r)-1}$ of distinct sets of cardinality $k$, such that no distinct $r$ elements in this collection form a sunflower.  If ${\bf A}_1,\dots,{\bf A}_r$ are drawn uniformly from this collection, then the probability that they form a sunflower is then precisely $(\Sun(k,r)-1)^{1-r}$ (the probability that all the ${\bf A}_i$ coincide).  Thus the bound of $(4 \Sun(k,r))^{-r}$ in the above lemma cannot be dramatically improved.

Lemma \ref{per} lets us control square functions:

\begin{corollary}[Sunflower bound on square function]\label{sunsquare} Let $S$ be a finite set, let $k \in \Z_+$, let $X$ be a measure space, and let $H$ be a finite-dimensional Hilbert space. Let $(f_A)_{A \in \binom{S}{k}}$ be a finite collection of functions $f_A \in L^{2r}(X;H)$.  Then
$$
\left\| (f_A)_{A \in \binom{S}{k}} \right\|_{L^{2r}(X; H^{[k]})}^{2r} 
\leq (4 \Sun(k,r))^{r} \sum_{A_0 \in \binom{S}{\leq k}} \sum^{**} \left\| \prod_{i=1}^r \|f_{A_0 \cup A_i}\|_H \right\|_{L^2(X)}^2$$
and conversely
$$
\sum_{A_0 \in \binom{S}{\leq k}} \sum^{**} \left\| \prod_{i=1}^r \|f_{A_0 \cup A_i}\|_H \right\|_{L^2(X)}^2 \leq \left\| (f_A)_{A \in \binom{S}{k}} \right\|_{L^{2r}(X; H^{[k]})}^{2r},
$$
where  $\sum^{**}$ denotes the sum over tuples $(A_1,\dots,A_r)$ of sets $A_1,\dots,A_r \in \binom{S \backslash A_0}{k-|A_0|}$ that are pairwise disjoint (or equivalently, that $A_0 \cup A_1,\dots,A_0 \cup A_r$ form a sunflower), and $H^{[k]} \coloneqq H^{\binom{S}{k}}$.
\end{corollary}

See \cite[Lemma 2.3]{IW} for a version of this result in the $k=1$ case.

\begin{proof}  We begin with the first inequality. Expanding out both sides, it suffices to establish the pointwise estimate
\begin{equation}\label{water}
 \left(\sum_{A \in \binom{S}{k}} \|f_A(x)\|_H^2\right)^r \leq
(4 \Sun(k,r))^{r} \sum_{A_0 \in \binom{S}{\leq k}} \sum^{**} \prod_{i=1}^r \|f_{A_0 \cup A_i}(x)\|_H^2
\end{equation}
for all $x$.

Fix $x$.  We may normalise the left-hand side of \eqref{water} to equal $1$.  We can then view the sequence $(\|f_A(x)\|_H^2)_{A \in \binom{S}{k}}$ as the probability density function for a random subset ${\bf A}$ of $S$ of cardinality $k$, and the inequality then can be written as
$$ 1 \leq (4 \Sun(k,r))^{r} \P( {\bf A}_1,\dots,{\bf A}_r \hbox{ form a sunflower} ).$$
The claim now follows from Lemma \ref{per}.  The second inequality similarly follows from the trivial bound
$$ \P( {\bf A}_1,\dots,{\bf A}_r \hbox{ form a sunflower} ) \leq 1.$$
\end{proof}

\section{Proof of main theorems}\label{proofs-sec}

Let $(d,k,S,\eps)$ be a major arc parameter set.  We now explore the additive structure of the major arcs associated to this set.  We first observe from the Chinese remainder theorem (and the hypothesis that the elements of $S$ are pairwise coprime) that
\begin{equation}\label{sigma}
\Sigma_{A_1} + \Sigma_{A_2} = \Sigma_{A_1 \uplus A_2}
\end{equation}
whenever $A_1,A_2 \subseteq S$ are disjoint. For $A_0 \in \binom{S}{\leq k}$, we also define the complementary set
$$ \Sigma_{(A_0)} \coloneqq \bigcup_{A \in \binom{S \backslash A_0}{\leq k-|A_0|}} \Sigma_A.$$
From \eqref{sigma} we then have the inclusion
\begin{equation}\label{sigma-include}
 \Sigma_{A_0} + \Sigma_{(A_0)} \subseteq \Sigma_{\leq k}.
\end{equation}

\begin{example} If $S = \{ q_1, q_2, q_3\}$ and $k=2$, then one has
\begin{align*}
 \Sigma_{(\emptyset)} &= \Sigma_{\leq 2} = \T^d[q_1 q_2] \cup \T^d[q_1 q_3] \cup \T^d[q_2 q_3], \\
 \Sigma_{(\{q_1\})} &= \Sigma_{\emptyset} \cup \Sigma_{\{q_2\}} \cup \Sigma_{\{q_3\}} = \T^d[q_2] \cup \T^d[q_3] \\
\Sigma_{(\{q_1,q_2\})} &= \Sigma_{\emptyset} = \{0\}.
\end{align*}
\end{example}

Let $H$ be a finite dimensional Hilbert space.  Define a \emph{major arc system} adapted to $(d,k,S,\eps)$ taking values in $H$ to be a collection $(f_\alpha)_{\alpha \in \Sigma_{\leq k}}$ of functions $f_\alpha \in \ell^2(\Z^d;H)$ with Fourier support in $\alpha + [-\eps,\eps]^d$ for each $\alpha \in \Sigma_{\leq k}$, thus
$$ f_\alpha \in \ell^2(\Z^d;H)^{\alpha + [-\eps,\eps]^d}.$$
For any $\Sigma \subseteq \Sigma_{\leq k}$, we define
$$ f_\Sigma \coloneqq \sum_{\alpha \in \Sigma} f_\alpha.$$

\begin{lemma}[Orthogonality properties]\label{ortho}  Let $(f_\alpha)_{\alpha \in \Sigma_{\leq k}}$ be a major arc system adapted to
a major arc parameter set $(d,k,S,\eps)$, taking values in a Hilbert space $H$. Suppose that the parameter set $(d,k,S,\eps)$ is $(r,c)$-good for some $r \in \Z_+$ and $0 < c < 1$.
\begin{itemize}
\item[(i)]  The major arcs $\alpha + [-\eps,\eps]^d$, $\alpha \in \Sigma_{\leq k}$ are disjoint. (Indeed, the $\alpha \in \Sigma_{\leq k}$ are at least $2\eps/c$-separated in the $\ell^\infty$ metric.)
\item[(ii)]  (Denominator orthogonality)  The hypersystem $(f_{\Sigma_A})_{A \in \binom{S}{\leq k}}$ is Type II $2r$-superorthogonal.
\item[(iii)]  (Numerator orthogonality)  If $A_1,\dots,A_r \in \binom{S}{\leq k}$ form a sunflower with core $A_0$ and petals $A_1 \backslash A_0,\dots,A_r \backslash A_0$, then the functions
$$ \prod_{i=1}^r f_{\Sigma_{A_0}+ \alpha_i} \in \ell^2(\Z^d; H^{\otimes r})$$
for $\alpha_1 \in \Sigma_{A_1 \backslash A_0}, \dots, \alpha_r \in \Sigma_{A_r \backslash A_0}$ are pairwise orthogonal in the Hilbert space $L^2(\T^d; H^{\otimes r})$, where we use the product notation
$$ \prod_{i=1}^r f_i(x) \coloneqq f_1(x) \otimes \dots \otimes f_r(x).$$
\end{itemize}
\end{lemma}

\begin{proof}
Note from Definition \ref{major-arc} that the coordinates of every element of $\Sigma_{\leq k}$ are rational numbers with denominator at most $q_{\max}^k$.  In particular, if $\alpha, \alpha'$ are two distinct elements of $\Sigma_{\leq k}$, then $\alpha, \alpha'$ differ in $\ell^\infty$ metric by at least $\frac{1}{q_{\max}^k}$.  The claim (i) now follows (with room to spare) from \eqref{small}.

Now we prove (ii).  From inspecting the Fourier transform, it suffices to show that
$$ \sum_{j=1}^r (\alpha_j + \theta_j) - \sum_{j=r+1}^{2r} (\alpha_j+\theta_j) \neq 0$$
in $\T^d$ whenever $\alpha_j \in \Sigma_{A_j}$ and $\theta_j \in [-\eps,\eps]^d$ for $j=1,\dots,2r$.  As the $A_1,\dots,A_{2r}$ contain a nonce, there exists an $A \subseteq S$ which contains all but exactly one of the $A_1,\dots,A_{2r}$.  Recalling that $Q_A = \prod_{q \in A} q$ and that the elements of $S$ are pairwise coprime, we conclude that $Q_A(\alpha_1+\dots+\alpha_r-\alpha_{r+1}-\dots-\alpha_{2r})$ has precisely one non-zero term in $\T^d$, and hence the point $\alpha_1+\dots+\alpha_r-\alpha_{r+1}-\dots-\alpha_{2r} \in \T^d$ is non-zero.  Observe that the coordinates of this point consist of rational numbers of denominator at most $\frac{1}{q_{\max}^{2rk}}$.  The claim now follows from \eqref{small} and the triangle inequality.

Now we prove (iii).  Inspecting the Fourier transform, it suffices to show that
$$ \sum_{j=1}^r (\alpha_{0,j} + \alpha_j + \theta_j) - \sum_{j=r+1}^{2r} (\alpha_{0,j} + \alpha_j+\theta_j) \neq 0$$
in $\T^d$ whenever $\alpha_{0,j} \in \Sigma_{A_0}$ and $\theta_j \in [-\eps,\eps]^d$ for $j=1,\dots,2r$, and
$$ (\alpha_1,\dots,\alpha_r), (\alpha_{r+1},\dots,\alpha_{2r}) \in \Sigma_{A_1 \backslash A_0} \times \dots \times \Sigma_{A_r \backslash A_0}$$
are distinct.  Multiplying by $Q_{A_0}$ to cancel the $\alpha_{0,j}$ factors, it suffices to show that
$$ \sum_{j=1}^r Q_{A_0}(\alpha_j + \theta_j) - \sum_{j=r+1}^{2r} Q_{A_0}(\alpha_j+\theta_j) \neq 0.$$
Note that $Q_{A_0} \leq q_{\max}^{|A_0|}$.  On the other hand, from the sunflower hypothesis and the Chinese remainder theorem we see that the point $Q_{A_0}(\alpha_1+\dots+\alpha_r - \alpha_{r+1}-\dots-\alpha_{2r})$ is non-zero.  The coordinates of this point consist of rational numbers of denominator at most $\frac{1}{q_{\max}^{2r(k-|A_0|)}}$, and the claim now follows from \eqref{small} and the triangle inequality.
\end{proof}

We can exploit these orthogonality properties to obtain a description of the $\ell^{2r}$ norm of a sum $f_{\Sigma_{\leq k}}$ associated to a major arc system, as well as a companion result that will be useful in the sequel.

\begin{theorem}[Applying orthogonality]\label{ortho-apply}  Let $(d,k,S,\eps)$ be a major arc parameter set which is $(r,c)$-good for some $r \in \Z_+$ and $0 < c < 1$. Let $H$ be a finite-dimensional Hilbert space.
\begin{itemize}
\item[(i)]  (Description of $\ell^{2r}$ norm)  If $(f_\alpha)_{\alpha \in \Sigma_{\leq k}}$ is a major arc system adapted to $(d,k,S,\eps)$, then we have
\begin{equation}\label{x1}
\begin{split}
 O(r \Log^{1/2}(k))^{-k}&\| f_{\Sigma_{\leq k}} \|_{\ell^{2r}(\Z^d;H)}\\
&\leq \left( \sum_{A_0 \in \binom{S}{\leq k}} \left\| (f_{\alpha + \Sigma_{A_0}})_{\alpha \in \Sigma_{(A_0)}} \right\|_{\ell^{2r}(\Z^d; H^{\Sigma_{(A_0)}})}^{2r}\right)^{1/2r} \\
&\leq O_c(1)^{d} O(1)^k \| f_{\Sigma_{\leq k}} \|_{\ell^{2r}(\Z^d;H)}.
\end{split}
\end{equation}
\item[(ii)]  (Rubio de Francia type estimate) Let $\varphi_0 \in C^\infty_c(\R)$ be a bump function supported on $[-1,1]$, and let $\varphi \in C^\infty_c(\R^d)$ be the symbol
$$ \varphi(\xi_1,\dots,\xi_d) = \prod_{j=1}^d \varphi_0\left(\frac{\xi_j}{\eps}\right).$$
Then for any $2 \leq p \leq \infty$ and $f \in \ell^p(\Z^d;H)$, one has the inequality
\begin{equation}\label{stab}
\left( \sum_{A_0 \in \binom{S}{\leq k}} \left\| (\Op_{\varphi; \alpha+\Sigma_{A_0}} f)_{\alpha \in \Sigma_{(A_0)}} \right\|_{\ell^p(\Z^d; H^{\Sigma_{(A_0)}})}^{p}\right)^{1/p}
\leq O_{\varphi_0}(1)^d O(1)^k \|f\|_{\ell^{p}(\Z^d;H)}.
\end{equation}
\end{itemize}
\end{theorem}

\begin{proof}  We begin with (ii), as this will be used in the proof of (i).
 By interpolation it suffices to establish the claims for $p=2,\infty$.  For $p=2$ the claim follows from Lemma \ref{ortho}(i) and Plancherel's theorem, noting that each $\Op_{\varphi; \alpha+\Sigma_{A_0}} f$ has Fourier transform supported in $\Sigma_{A_0} + \alpha + [-\eps,\eps]^d$, and each $\beta \in \Sigma_{\leq k}$ has at most $2^k$ representations of the form $\beta = \alpha+\alpha_0$ with $A_0 \in \binom{S}{\leq k}$, $\alpha_0 \in \Sigma_{A_0}$, $\alpha \in \Sigma_{(A_0)}$.  For $p=\infty$, it suffices by translation invariance to show that
$$
\left\|\left(\Op_{\varphi; \alpha+\Sigma_{A_0}} f(0)\right)_{\alpha \in \Sigma_{(A_0)}} \right\|_{H^{\Sigma_{(A_0)}}}
\leq O_{\varphi_0}(1)^d O(1)^k \|f\|_{\ell^\infty(\Z^d;H)}$$
for any $f \in L^\infty(\T^d;H)$ and $A_0 \in \binom{S}{\leq k}$.  
By the inclusion-exclusion formula, and conceding a factor of $2^k$, it suffices to show that
$$
\left\|\left( \Op_{\varphi; \alpha+\Sigma_{\subseteq A'_0}} f(0) \right)_{\alpha \in \Sigma_{(A_0)}} \right\|_{H^{\Sigma_{(A_0)}}}
\leq O_{\varphi_0}(1)^{d} \|f\|_{\ell^\infty(\Z^d;H)}$$
for any $A'_0 \subseteq A_0$. By duality, this bound is equivalent to the assertion that
$$
\left\| \sum_{\alpha \in \Sigma_{(A_0)}} c_\alpha \Op_{\varphi; \alpha+\Sigma_{\subseteq A'_0}}^* \delta_0 \right\|_{\ell^1(\Z^d;H)} \leq O_{\varphi_0}(1)^{d} \| (c_\alpha)_{\alpha \in \Sigma_{(A_0)}} \|_{H^{\Sigma_{(A_0)}}}$$
for any sequence $c_\alpha \in H$, $\alpha \in \Sigma_{(A_0)}$, where $\delta_0$ is the Kronecker delta function.  Observe that the integrand on the left-hand side is actually supported on $(Q_{A'_0}\Z)^d$.  If we introduce the weight function
$$ w(n_1,\dots,n_d) \coloneqq \prod_{j=1}^d (1+\eps^2 n_j^2)$$
we see from \eqref{small} that
$$ \| w^{-1} \|_{\ell^2(\Z^d)} \leq O(1)^d \eps^{d/2} Q_{A'_0}^{d/2} $$
so by Cauchy--Schwarz it will suffice to establish the bound
$$
\left\| \sum_{\alpha \in \Sigma_{(A_0)}} c_\alpha w \Op_{\varphi; \alpha+\Sigma_{\subseteq A'_0}}^* \delta_0 \right\|_{\ell^2(\Z^d;H)} \leq O_{\varphi_0}(1)^{d} \eps^{d/2} Q_{A'_0}^{d/2} \| (c_\alpha)_{\alpha \in \Sigma_{(A_0)}} \|_{H^{\Sigma_{(A_0)}}}$$
From Lemma \ref{ortho}(i) we see that the functions $w \Op_{\varphi; \alpha+\Sigma_{\subseteq A'_0}}^* \delta_0$, $\alpha \in \Sigma_{(A_0)}$ have disjoint Fourier supports and are thus pairwise orthogonal in $\ell^2(\Z^d;H)$.  Thus by the Pythagorean theorem, it suffices to show that
$$
\left\| w \Op_{\varphi; \alpha+\Sigma_{\subseteq A'_0}}^* \delta_0 \right\|_{\ell^2(\Z^d;H)} \leq O_{\varphi_0}(1)^{d} \eps^{d/2} Q_{A'_0}^{d/2} $$
for each $\alpha \in \Sigma_{(A_0)}$.  Each function $w \Op_{\varphi; \alpha+\Sigma_{\subseteq A'_0}}^* \delta_0$ can be split further in $\ell^2(\Z^d;H)$ into $Q_{A'_0}^d$ orthogonal components $w \Op_{\varphi; \alpha+\alpha_0}^* \delta_0$ with $\alpha_0 \in \Sigma_{\subseteq A'_0}$.  By the Pythagorean theorem again, it thus suffices to establish the bound
$$
\| w \Op_{\varphi; \alpha+\alpha_0}^* \delta_0 \|_{\ell^2(\Z^d)} \leq O_{\varphi_0}(1)^{d} \eps^{d/2} $$
for each $\alpha \in \Sigma_{(A_0)}$, $\alpha_0 \in \Sigma_{\subseteq A'_0}$.  The magnitude of the expression inside the norm of the left-hand side does not actually depend on $\alpha+\alpha_0$, so we may assume that $\alpha+\alpha_0=0$.  The left-hand side then factors as a tensor product and it now suffices to establish the claim for $d=1$, that is to say to show that
$$ \sum_{n \in \Z} (1 + \eps^2 n^2)^2 \eps^2 |\hat \varphi_0(\eps n)|^2 \leq O_{\varphi_0}(\eps)$$
which follows from the rapid decrease of $\hat \varphi$ (and noting from \eqref{small} that $\eps \leq 1$).  This completes the proof of (ii).

Now we prove (i).  By Lemma \ref{ortho}(ii) and Theorem \ref{sys-ortho}(ii), we have
$$ 
\| f_{\Sigma_{\leq k}} \|_{\ell^{2r}(\Z^d;H)} \leq O(r)^{k/2} \left\| (f_{\Sigma_A})_{A \in \binom{S}{\leq k}} \right\|_{\ell^{2r}(\Z^d; H^{[\leq k]})}.$$
If we then apply Corollary \ref{sunsquare} with the sunflower bound \eqref{rao}, and take $(2r)^{\mathrm{th}}$ roots, we obtain
$$ 
\| f_{\Sigma_{k'}} \|_{\ell^{2r}(\Z^d;H)} \leq O(r)^{k/2} O(r \Log(k'r))^{k'/2} \left\| (f_{\Sigma_A})_{A \in \binom{S}{k'}} \right\|_{\ell^{2r}(\Z^d; H^{\binom{S}{k'}})}$$
for any $k' \leq k$.  Summing in $k$, we conclude that
\begin{equation}\label{cake}
\| f_{\Sigma_{\leq k}} \|_{\ell^{2r}(\Z^d;H)} \leq O(r \Log^{1/2}(kr))^k \left( \sum_{A_0 \in \binom{S}{\leq k}} X_{A_0}\right)^{1/2r}
\end{equation}
where
$$ X_{A_0} \coloneqq  \sum^{***} \| \prod_{i=1}^r f_{\Sigma_{A_0 \cup A_i}} \|_{\ell^2(\Z^d;H^{\otimes r})}^2$$
and $\sum^{***}$ denotes a sum over tuples $(A_1,\dots,A_r)$ of disjoint sets $A_1,\dots,A_r \in \binom{S \backslash A_0}{\leq k-|A_0|}$.
For $A_0,A_1,\dots,A_r$ as above, we can split
$$ \prod_{i=1}^r f_{\Sigma_{A_0 \cup A_i}} = \sum_{\alpha_1 \in A_1,\dots,\alpha_r \in A_r} \prod_{i=1}^r f_{\alpha_i + \Sigma_{A_0}}.$$
From Lemma \ref{ortho}(iii) and the Pythagorean theorem, we may thus write 
$$
X_{A_0} =\sum^{***} \sum_{\alpha_1 \in \Sigma_{A_1},\dots,\alpha_r \in \Sigma_{A_r}} \| \prod_{i=1}^r f_{\alpha_i + \Sigma_{A_0}} \|_{\ell^2(\Z^d;H^{\otimes r})}^2.$$
We drop the hypothesis of disjointness in the $\sum^{***}$ sum to obtain the upper bound
$$
X_{A_0} \leq \sum_{A_1,\dots,A_r \in \binom{S \backslash A_0}{\leq k-|A_0|}} \sum_{\alpha_1 \in \Sigma_{A_1},\dots,\alpha_r \in \Sigma_{A_r}} \| \prod_{i=1}^r f_{\alpha_i + \Sigma_{A_0}} \|_{\ell^2(\Z^d;H^{\otimes r})}^2$$
which by the Fubini--Tonelli theorem can be rearranged as
$$
X_{A_0} \leq \left\| (f_{\alpha + \Sigma_{A_0}})_{\alpha \in \Sigma_{(A_0)}} \right\|_{\ell^{2r}(\Z^d; H^{\Sigma_{(A_0)}})}^{2r}.$$
This gives the first inequality in \eqref{x1}.  

Now we establish the second inequality in \eqref{x1}.  Let $c' \coloneqq \frac{1+c}{2}$, so that $c < c' < 1$. Let $\varphi \in C^\infty_c(\R^d)$ be a multiplier of the form
\begin{equation}\label{varph}
 \varphi(\xi_1,\dots,\xi_d) \coloneqq \prod_{j=1}^d \varphi_0(\xi_j/\eps),
\end{equation}
where $\varphi_0 \in C^\infty_c(\R)$ is a fixed real even bump function (depending only on $c$) supported on $[-c'/c,c'/c]$ that equals $1$ on $[-1,1]$.  From Lemma \ref{ortho}(i) (with $c$ replaced by $c'$, and $\eps$ replaced by $\frac{c'}{c} \eps$) we have
$$ f_{\alpha + \Sigma_{A_0}} \coloneqq \Op_{\varphi; \alpha+\Sigma_{A_0}} f_{\Sigma_{\leq k}}$$
and the claim now follows from (ii) (setting $p=2r$).
\end{proof}

Now we can prove Theorem \ref{main}.  Let the notation and hypotheses be as in that theorem.  We normalize $\|\Op_m\|_{B(L^{2r}(\R^d))}=1$ and $\|f\|_{\ell^{2r}(\Z^d;H)}=1$ (we can also assume by limiting arguments that $f \in \ell^2(\Z^d;H)$ to avoid technicalities), and our task is to show that
$$
\left\| \sum_{A \in \binom{S}{\leq k}} \epsilon_A \Op_{m;\Sigma_A} f\right \|_{\ell^{2r}(\Z^d;H)}
\leq O_c(1)^d O(r \Log^{1/2}(k))^k.
$$
Applying Theorem \ref{ortho-apply}(i) to the hypersystem $(\epsilon_A \Op_{m;\Sigma_A} f)_{A \in \binom{S}{\leq k}}$, it suffices to show that
\begin{equation}\label{maz}
 \left( \sum_{A_0 \in \binom{S}{\leq k}} \left\| (\Op_{m;\alpha + \Sigma_{A_0}} f)_{\alpha \in \Sigma_{(A_0)}}\right \|_{\ell^{2r}(\Z^d; H^{\Sigma_{(A_0)}})}^{2r}\right)^{1/2r} \leq O_c(1)^{d+k}.
\end{equation}
With $\varphi$ as in \eqref{varph}, we may use Lemma \ref{ortho}(i) to factor
$$ \Op_{m;\alpha + \Sigma_{A_0}} f= \Op_{m;\alpha+\Sigma_{A_0}} \Op_{\varphi;\alpha+\Sigma_{A_0}} f.$$
Next, from the Magyar--Stein--Wainger sampling principle (Proposition \ref{msw}) we have
$$ \| \Op_{m;\Sigma_{A_0}} F \|_{\ell^{2r}(\Z^d;H)} \leq O(1)^d \|F\|_{\ell^{2r}(\Z^d;H)}$$
for any $F \in \ell^{2r}(\Z^d;H)$, hence by the Marcinkiewicz--Zygmund theorem (Theorem \ref{mz}) one has
$$ \left\| (\Op_{m;\Sigma_{A_0}} F_\alpha)_{\alpha \in \Sigma_{(A_0)}} \right\|_{\ell^{2r}(\Z^d; H^{\Sigma_{(A_0)}})} \leq O(1)^d \left\|(F_\alpha)_{\alpha \in \Sigma_{(A_0)}}  \right\|_{\ell^{2r}(\Z^d; H^{\Sigma_{(A_0)}})}$$
for any $F_\alpha \in \ell^{2r}(\Z^d;H)$, which by the modulation symmetries of the Fourier transform imply that
$$ \left\| (\Op_{m;\alpha+\Sigma_{A_0}} F_\alpha)_{\alpha \in \Sigma_{(A_0)}} \right\|_{\ell^{2r}(\Z^d; H^{\Sigma_{(A_0)}})} \leq O(1)^d \left\|(F_\alpha)_{\alpha \in \Sigma_{(A_0)}} \right\|_{\ell^{2r}(\Z^d; H^{\Sigma_{(A_0)}})}.$$
Putting all this together, we reduce to showing that
$$
 \left( \sum_{A_0 \in \binom{S}{\leq k}} \left\| (\Op_{\varphi;\alpha + \Sigma_{A_0}} f)_{\alpha \in \Sigma_{(A_0)}}\right \|_{\ell^{2r}(\Z^d; H^{\Sigma_{(A_0)}})}^{2r}\right)^{1/2r} \leq O_c(1)^d O(1)^k.
$$
But this follows from Theorem \ref{ortho-apply}(ii).  This concludes the proof of Theorem \ref{main}.

Now we observe an arithmetic analogue of Theorem \ref{ortho-apply}, in which the spatial scale parameter $\eps$ becomes irrelevant:

\begin{theorem}[Applying orthogonality, arithmetic limit]\label{ortho-apply-arith}  Let $(d,k,S,\eps)$ be a major arc parameter set. Let $H$ be a finite-dimensional Hilbert space.  For each $\alpha \in \Sigma_{\leq k}$, let $f_\alpha \in \Schwartz(\A_\Z^d)$ have Fourier support in $\R^d \times \{\alpha\}$, and define $f_\Sigma \coloneqq \sum_{\alpha \in \Sigma} f_\alpha$ as before.  Then for every positive integer $r$, we have
\begin{equation}\label{x2}
\begin{split}
 O(r \Log^{1/2}(k))^{-k}&\| f_{\Sigma_{\leq k}} \|_{L^{2r}(\A_\Z^d;H)}\\
&\leq \left( \sum_{A_0 \in \binom{S}{\leq k}} \left\| (f_{\alpha + \Sigma_{A_0}})_{\alpha \in \Sigma_{(A_0)}} \right\|_{L^{2r}(\A_\Z^d; H^{\Sigma_{(A_0)}})}^{2r}\right)^{1/2r} \\
&\leq O(1)^{d+k} \| f_{\Sigma_{\leq k}} \|_{L^{2r}(\A_\Z^d;H)}.
\end{split}
\end{equation}
Also, we have
\begin{equation}\label{x3}
\| \Op_{1_{\Sigma_{\leq k}}} \|_{B(L^{2r}(\hat \Z^d;H))} \leq O(1)^d O(r \Log^{1/2}(k))^{k}.
\end{equation}
\end{theorem}

\begin{proof}  One can establish \eqref{x2} by direct repetition of the proof of Theorem \ref{ortho-apply}, but we shall instead deduce this theorem as a limiting case of Theorem \ref{ortho-apply} (basically by sending $\eps$ to zero).  By splitting $\A_\Z^d$ into fibres $\{x\} \times \hat \Z^d$ for $x \in \R^d$ and using the Fubini--Tonelli theorem, it suffices to establish the analogous claim for $\hat \Z^d$, that is to say to establish the bound
\begin{equation}\label{x4}
\begin{split}
 O(r \Log^{1/2}(k))^{-k}&\| f_{\Sigma_{\leq k}} \|_{L^{2r}(\hat \Z^d;H)}\\
&\leq \left( \sum_{A_0 \in \binom{S}{\leq k}} \left\| (f_{\alpha + \Sigma_{A_0}})_{\alpha \in \Sigma_{(A_0)}} \right\|_{L^{2r}(\hat \Z^d; H^{\Sigma_{(A_0)}})}^{2r}\right)^{1/2r} \\
&\leq O(1)^{d+k} \| f_{\Sigma_{\leq k}} \|_{L^{2r}(\hat \Z^d;H)}
\end{split}\end{equation}
where for each $\alpha \in \Sigma_{\leq k}$, $f_\alpha \in \Schwartz(\hat \Z^d;H)$ has Fourier support in $\{\alpha\}$, that is to say $f_\alpha(y) = c_\alpha e(-y \cdot \alpha)$ for some $c_\alpha \in H$.  Let $0 < \eps < 1$ be a small parameter, let $\varphi \in \Schwartz(\R^d)$ be a Schwartz function whose Fourier transform is supported in $[-1,1]^d$ with normalization $\|\varphi\|_{L^{2r}(\R^d)}=1$, and let $f_{\alpha,\eps} \in \Schwartz(\Z^d;H)$ be the functions
$$ f_{\alpha,\eps}(n) \coloneqq f_\alpha(\hat \iota(n)) \varphi(\eps n ) = c_\alpha e(-n \cdot \alpha) \varphi( \eps n )$$
where $\hat \iota \colon \Z^d \to \hat \Z^d$ is the canonical embedding.  Then $(f_{\alpha,\eps'})_{\alpha \in \Sigma_{\leq k}}$ is a major arc system adapted to $(d,k,S,\eps)$.  For $\eps$ small enough, this set of parameters is $(r,1/2)$-good, and so we see from Theorem \ref{ortho-apply} that
\begin{equation}\label{sides}
\begin{split}
 O(r \Log^{1/2}(k))^{-k}&\| f_{\Sigma_{\leq k},\eps} \|_{\ell^{2r}(\Z^d;H)}\\
&\leq \left( \sum_{A_0 \in \binom{S}{\leq k}} \left\| (f_{\alpha + \Sigma_{A_0},\eps})_{\alpha \in \Sigma_{(A_0)}} \right\|_{L^{2r}(\Z^d; H^{\Sigma_{(A_0)}})}^{2r}\right)^{1/2r} \\
&\leq O(1)^{d+k} \| f_{\Sigma_{\leq k},\eps} \|_{\ell^{2r}(\Z^d;H)}
\end{split}
\end{equation}
where for any $\Sigma \subseteq \Sigma_{\leq k}$ we denote
$$f_{\Sigma,\eps}(n) \coloneqq \sum_{\alpha \in \Sigma} f_{\alpha,\eps}(n) = f_\Sigma(\hat \iota(n)) \varphi(\eps n).$$
The functions $f_\Sigma \circ \hat \iota$ are all periodic with period $Q_S$.  By Riemann integrability one then has
$$ \eps^{1/2r} \| f_{\Sigma_{\leq k},\eps} \|_{\ell^{2r}(\Z^d;H)} \to \| f_{\Sigma_{\leq k}} \|_{\ell^{2r}(\hat \Z^d;H)}$$
and similarly
$$ \eps^{1/2r} \left\| (f_{\alpha + \Sigma_{A_0},\eps})_{\alpha \in \Sigma_{(A_0)}} \right\|_{L^{2r}(\Z^d; H^{\Sigma_{(A_0)}})} \to
\left\| (f_{\alpha + \Sigma_{A_0}})_{\alpha \in \Sigma_{(A_0)}} \right\|_{L^{2r}(\hat \Z^d; H^{\Sigma_{(A_0)}})}$$
as $\eps' \to 0$ for any $A_0$.  Multiplying \eqref{sides} by $\eps^{1/2r}$ and taking the limit $\eps \to 0$, we obtain the claim \eqref{x2}.

We now prove \eqref{x3}.  Let $F \in L^2(\hat \Z^d;H)$, then we have $\Op_{1_{\Sigma_{\leq k}}} F = F_{\Sigma_{\leq k}}$, where $F_\alpha(y) \coloneqq e(-y \cdot \alpha) \F_{\hat \Z^d} F(\alpha)$ and $F_\Sigma \coloneqq \sum_{\alpha \in \Sigma} F_\alpha$ for any $\Sigma \subseteq \Sigma_{\leq k}$.  By \eqref{x2} we then have
$$ \| \Op_{1_{\Sigma_{\leq k}}} F\|_{L^{2r}(\hat \Z^d;H)} \leq O(r \Log^{1/2}(k))^{k} \left( \sum_{A_0 \in \binom{S}{\leq k}}\left\| (F_{\alpha + \Sigma_{A_0}})_{\alpha \in \Sigma_{(A_0)}} \right\|_{L^{2r}(\hat \Z^d; H^{\Sigma_{(A_0)}})}^{2r}\right)^{1/2r} $$
so it will suffice to establish the bound
$$
\left( \sum_{A_0 \in \binom{S}{\leq k}} \left\| (F_{\alpha + \Sigma_{A_0}})_{\alpha \in \Sigma_{(A_0)}} \right\|_{L^p(\hat \Z^d; H^{\Sigma_{(A_0)}})}^p\right)^{1/p} \leq O(1)^d \|F\|_{L^p(\hat \Z^d;H)}$$
for all $2 \leq p \leq \infty$.  By interpolation it suffices to establish this for $p=2$ and $p=\infty$.  The claim $p=2$ is immediate from Bessel's inequality.  For $p=\infty$ it suffices by translation invariance to show that
$$(\sum_{\alpha \in \Sigma_{(A_0)}} \|F_{\alpha + \Sigma_{A_0}}(0)\|_H^2)^{1/2} \leq O(1)^d \|F\|_{L^\infty(\hat \Z;H)}$$
which by duality is equivalent to the assertion that
$$
\int_{\hat \Z^d} \left\|\sum_{\alpha \in \Sigma_{(A_0)}} c_\alpha \sum_{\alpha_0 \in \Sigma_{A_0}} e(-y \cdot (\alpha+\alpha_0))\right\|_H\ d\mu_{\hat \Z^d}(y)
\leq O(1)^d (\sum_{\alpha \in \Sigma_{(A_0)}} \|c_\alpha\|_H^2)^{1/2}$$
for any $c_\alpha \in H$ for $\alpha \in \Sigma_{(A_0)}$.

Observe that the integrand vanishes unless the projection of $y$ to $(\Z/Q_{A_0}\Z)^d$ vanishes, thus the integrand is supported on a set of measure $Q_{A_0}^{-d}$.  By Cauchy-Schwarz, it thus suffices to show that
$$
\int_{\hat \Z^d} \left\|\sum_{\alpha \in \Sigma_{(A_0)}} c_\alpha \sum_{\alpha_0 \in \Sigma_{A_0}} e(-y \cdot (\alpha+\alpha_0))\right\|_H^2\ d\mu_{\hat \Z^d}(y)
\leq O(1)^d Q_{A_0}^d \sum_{\alpha \in \Sigma_{(A_0)}} \|c_\alpha\|_H^2.$$
But this is immediate from Plancherel's theorem since $|\Sigma_{A_0}| = Q_{A_0}^d$.
\end{proof}

Now we can prove Theorem \ref{main-sampling}.  To abbreviate the notation we write $X \lessapprox Y$ for
$$ X \leq \exp( O_c( d) + O(k\Log (r\Log k) ) ) Y$$
and $X \approx Y$ for $X \lessapprox Y \lessapprox X$.

We first establish the claim in the case $p=2r$. By a limiting argument we may assume that $f \in \ell^{2}(\Z^d;H)^{\pi(\Omega)}$, thus we can write $f = \sum_{\alpha \in \Sigma_{\leq k}} f_\alpha$ where
$$ f_\alpha(n) \coloneqq \int_{[-\eps,\eps]^d} e(- n \cdot (\alpha+\theta)) \F_{\Z^d} f(\alpha + \theta)\ d\theta.$$
We then have $\Sample_\Omega^{-1} f = \sum_{\alpha \in \Sigma_{\leq k}} F_\alpha$ where
$$ F_\alpha(x,y) \coloneqq e(-y \cdot \alpha) \int_{[-\eps,\eps]^d} e(- x \cdot \theta) \F_{\Z^d} f(\alpha + \theta)\ d\theta.$$
Writing $f_\Sigma \coloneqq \sum_{\alpha \in \Sigma} f_\alpha$ and $F_\Sigma \coloneqq \sum_{\alpha \in \Sigma} F_\alpha$ for any $\Sigma \subseteq \Sigma_{\leq k}$, we see from Theorem \ref{ortho-apply}(i) (and bounding $r \Log^{1/2}(k) = \exp( O(\Log(r \Log k) ) )$) that
$$ \| f \|_{\ell^{2r}(\Z^d;H)} \approx \left( \sum_{A_0 \in \binom{S}{\leq k}} \| (f_{\alpha + \Sigma_{A_0}})_{\alpha \in \Sigma_{(A_0)}} \|_{\ell^{2r}(\Z^d; H^{\Sigma_{(A_0)}})}^{2r}\right)^{1/2r} $$
and similarly from Theorem \ref{ortho-apply-arith} that
$$\| \Sample_\Omega^{-1} f \|_{L^{2r}(\A_\Z^d;H)} \approx \left( \sum_{A_0 \in \binom{S}{\leq k}} \| (F_{\alpha + \Sigma_{A_0}})_{\alpha \in \Sigma_{(A_0)}} \|_{L^{2r}(\A_\Z^d; H^{\Sigma_{(A_0)}})}^{2r}\right)^{1/2r}
$$
so it will suffice to show that
$$
\| (f_{\alpha + \Sigma_{A_0}})_{\alpha \in \Sigma_{(A_0)}} \|_{\ell^{2r}(\Z^d; H^{\Sigma_{(A_0)}})} \approx \| (F_{\alpha + \Sigma_{A_0}})_{\alpha \in \Sigma_{(A_0)}} \|_{L^{2r}(\A_\Z^d; H^{\Sigma_{(A_0)}})}
$$
for each $A_0 \in \binom{S}{\leq k}$.  From expanding the definitions, we see that
$$ (F_{\alpha + \Sigma_{A_0}})_{\alpha \in \Sigma_{(A_0)}} \in L^2(\A_\Z^d; H^{\Sigma_{(A_0)}})^{[-\eps,\eps]^d \times \Sigma_{A_0}}$$
and
$$ (f_{\alpha + \Sigma_{A_0}})_{\alpha \in \Sigma_{(A_0)}} = \Sample (F_{\alpha + \Sigma_{A_0}})_{\alpha \in \Sigma_{(A_0)}}.$$
The claim now follows from Proposition \ref{qss} and Lemma \ref{ortho}(i).

Now we establish Theorem \ref{main-sampling} for general $(2r)' \leq p \leq 2r$.  We begin with the upper bound
$$
\| \Sample_\Omega^{-1} f \|_{L^p(\A_\Z^d;H)} \lessapprox \|f\|_{\ell^p(\Z^d;H)}$$
for $f \in \ell^p(\Z^d;H)^{\pi(\Omega)}$.  With $\varphi$ as in \eqref{varph}, we can write
$$\Sample_\Omega^{-1} f = \Sample_{\Omega'}^{-1} \Op_{\varphi; \Sigma_{\leq k}} f$$
where $\Omega \coloneqq [-\frac{c'}{c} \eps,\frac{c'}{c}\eps]^d \times \Sigma_{\leq k}$.  Note that the right-hand side is well defined for all $f$ in $\ell^p(\Z^d;H)$ (with no restriction on the Fourier support on $f$).  Thus it will suffice to show that
\begin{equation}\label{equiv-2}
\| \Sample_{\Omega'}^{-1} \Op_{\varphi; \Sigma_{\leq k}} \|_{B(\ell^p(\Z^d;H) \to L^p(\A_\Z^d;H))} \lessapprox 1
\end{equation}
for all $(2r)' \leq p \leq 2r$.  Now that the Fourier restriction has been removed, interpolation becomes available, and it suffices to establish this bound for $p=2r, (2r)'$.  For $p=2r$ the claim follows from the $p=2r$ case of Theorem \ref{main-sampling} already established (with $c$ replaced by $c'$), noting from Theorem \ref{main} that
\begin{equation}\label{opo}
\|\Op_{\varphi; \Sigma_{\leq k}} \|_{B(\ell^p(\Z^d;H))} \lessapprox 1
\end{equation}
for $p=2r$ (indeed, this estimate holds for all $(2r)' \leq p \leq 2r$ by duality and interpolation).

For $p=(2r)'$, we apply duality to write the estimate in the equivalent form
\begin{equation}\label{equiv}
\| \Op_{\varphi; \Sigma_{\leq k}} \Sample \|_{B(L^{2r}(\A_\Z^d;H) \to \ell^{2r}(\Z^d;H))} \lessapprox 1.
\end{equation}
From \eqref{sam} we have
$$ \Op_{\varphi; \Sigma_{\leq k}} \Sample  = \Sample \Op_{\varphi \otimes 1_{\Sigma_{\leq k}}}.$$
From Theorem \ref{ortho-apply-arith} one has
$$ \|\Op_{\varphi \otimes 1_{\Sigma_{\leq k}}}\|_{B(L^{2r}(\A_\Z^d;H))}
= \|\Op_{\varphi}\|_{B(L^{2r}(\R^d;H))}
\|\Op_{1_{\Sigma_{\leq k}}}\|_{B(L^{2r}(\hat \Z^d;H))} \lessapprox 1$$
and the claim now follows from the $p=2r$ case of Theorem \ref{main-sampling} already established (with $c$ replaced by $c'$).

Now we obtain the lower bound 
$$
\| \Sample_\Omega^{-1} f \|_{L^p(\A_\Z^d;H)} \gtrapprox \|f\|_{\ell^p(\Z^d;H)}$$
for $f \in \ell^p(\Z^d;H)^{\pi(\Omega)}$.  This is equivalent to
$$
\| \Sample F \|_{\ell^p(\Z^d;H)} \lessapprox \|F\|_{L^p(\A_\Z^d;H)}$$
for $F \in L^p(\Z^d;H)^{\Omega}$.  For such $F$ we have $\Sample F = \Sample \Op_{\varphi \otimes 1_{\Sigma_{\leq k}}} F$, so it suffices to show that
$$
\| \Sample \Op_{\varphi \otimes 1_{\Sigma_{\leq k}}} \|_{B(L^p(\A_\Z^d;H) \to \ell^p(\Z^d;H))} \lessapprox 1$$
for all $(2r)' \leq p \leq 2r$.  By interpolation it suffices to establish this bound for $p=2r, (2r)'$.  For $p=2r$ the claim follows from \eqref{equiv}.  For $p=(2r)'$ we dualize to
$$ \| \Op_{\varphi \otimes 1_{\Sigma_{\leq k}}} \Sample_{\Omega'}^{-1} \|_{B(\ell^{2r}(\Z^d;H) \to L^{2r}(\A_\Z^d;H))} \lessapprox 1$$
and the claim now follows from \eqref{equiv-2}, \eqref{sam}.  This concludes the proof of Theorem \ref{main-sampling} for general $p$.

Now we can prove Theorem \ref{main-adelic}.  Let the notation and hypotheses be as in that theorem, and as before let $\varphi$ be the function \eqref{varph}.  Then by \eqref{sam} (and Lemma \ref{ortho}(i)) we can factorize
$$ 
\Op_{m;\Sigma_{\leq k}} = \Op_{m;\Sigma_{\leq k}}  \Op_{\varphi;\Sigma_{\leq k}}  =\Sample \Op_m \Sample_{\Omega'}^{-1} \Op_{\varphi; \Sigma_{\leq k}}.$$
The claim now follows from \eqref{opo} and Theorem \ref{main-sampling}.

\section{The Ionescu--Wainger major arc construction}\label{num-sec}

We now describe the specific choice of major arcs that essentially appears in the original work \cite{IW} of Ionescu and Wainger, as well as in many subsequent works.

\begin{lemma}\label{iwac}  Let $0 < \rho < 1$ be a parameter, and set $k \coloneqq \lfloor \frac{2}{\rho} \rfloor + 1$.  Suppose that $N \geq 2^k$.  Then there exists a set $S$ of pairwise coprime natural numbers such that for any $d \in \Z_+$, and $\eps>0$, the major arc parameter set $(d,k,S,\eps)$ obeys the following properties:
\begin{itemize}
\item[(i)]  One has $\T^d[q] \subseteq \Sigma_{\leq k}$ for all natural numbers $1 \leq q \leq N$.
\item[(ii)]  One has $\Sigma_{\leq k} \subseteq \T^d[Q]$ for some $Q \leq 3^N$.
\item[(iii)]  All elements of $S$ are bounded by $C^{k N^{\rho/2}}$ for some absolute constant $C>1$. In particular, $(d,k,S,\eps)$ will be $(r,\frac{1}{2})$-good whenever
\begin{equation}\label{epso}
 \eps < \frac{1}{4r C^{2rk^2 N^{\rho/2}}}.
\end{equation}
\item[(iv)]  $\Sigma_{\leq k}$ is the union of finitely many subgroups of $\T^d$, each of the form $\T^d[q]$ for some $q \leq O(1)^{k^2 N^{\rho/2}}$.  In particular, $|\Sigma_{\leq k}| \leq O(1)^{dk^2 N^{\rho/2}}$.
\end{itemize}
\end{lemma}

Typically $\rho$ (and hence $k$) and $r$ will be fixed in applications.  For $N$ sufficiently large depending on $\rho,r,d$, the condition \eqref{epso} can be simplified to $\eps \leq \exp(-N^{\rho})$, and the bounds $q \leq O(1)^{k^2 N^{\rho/2}}$, $|\Sigma_{\leq k}| \leq O(1)^{dk^2 N^{\rho/2}}$ in (iv) can be similarly simplified to $q, |\Sigma_{\leq k}| \leq \exp(N^\rho)$.  The main point here is we can cover the Farey sequence $\bigcup_{1 \leq q \leq N} \T^d[q]$ by good major arcs whose width $\eps$ can be as large as $\exp( - N^\rho )$.

\begin{proof}  We set $S$ equal to
$$ S \coloneqq \{ \prod_{p \leq N^{\rho/2}} p^{\lfloor \frac{\log N}{\log p}\rfloor} \} \cup \{ p^{\lfloor \frac{\log N}{\log p}\rfloor}: N^{\rho/2} < p \leq N \}$$
where $p$ is always understood to be restricted to the primes.  Clearly the elements of $S$ are pairwise coprime.  To prove (i), we have to show that every natural number $1 \leq q \leq N$ is a factor of a product of at most $k$ distinct elements from $S$. But by the fundamental theorem of arithmetic we can write $q = p_1^{a_1} \dots p_m^{a_m}$ for some primes $1 < p_1 < \dots < p_m \leq N$ and $1 \leq a_i \leq \lfloor \frac{\log N}{\log p_i} \rfloor$.  At most $\lfloor \frac{2}{\rho} \rfloor = k-1$ of these primes can exceed $N^{\rho/2}$.  One can then write $q$ as a factor of $
\prod_{p \leq N^{\rho/2}} p^{\lfloor \frac{\log N}{\log p}\rfloor}$ times at most $k-1$ terms of the form $p^{\lfloor \frac{\log N}{\log p}\rfloor}$, giving the claim.

The product $Q_S$ of all the elements of $S$ is equal to
$$ \prod_{p \leq N} p^{\lfloor \frac{\log N}{\log p}\rfloor} = \mathrm{lcm}(1,\dots,N) \leq 3^N$$
where the latter inequality is established in \cite{hanson}.  Since $\Sigma_{\leq k} \subseteq \T^d[Q_S]$, this gives (ii).

For (iii), we trivially bound $p^{\lfloor \frac{\log N}{\log p}\rfloor}$ by $N$, and note from the prime number theorem that the number of primes less than $N^{\rho/2}$ is $O(N^{\rho/2} / \log N^{\rho/2} ) = O( k N^{\rho/2} / \log N )$, giving (iii) as claimed (noting from the hypothesis $N \geq 2^k$ that $N^{\rho/2} \geq 2^{1/2}$ and hence $N \leq N^{O(N^{\rho/2} / \log N^{\rho/2} )} = O(1)^{kN^{\rho/2}}$).

Finally to prove (iv), note from definition that $\Sigma_{\leq k}$ is the union of  $\T^d[q]$  where $q$ is the product of at most $k$ elements of $S$, and the claim now follows from (iii).
\end{proof}

As a particular corollary of this construction, we can prove a sampling theorem for the classical major arcs.

\begin{corollary}[Classical major arc sampling]\label{main-sampling-class}  Let $d, N \in \Z_+$, $\eps>0$, and set
$$\Omega \coloneqq [-\eps,\eps]^d \times \bigcup_{q=1}^N \T^d[Q].$$
Let $0 < \rho < 1$ and $1 < p < \infty$ be such that
\begin{equation}\label{eps-small}
\eps < \exp( - C \max(p,p') \rho^{-2} N^{\rho/2} )
\end{equation}
for a sufficiently large absolute constant $C$.  Then for any finite-dimensional Hilbert space $H$, one has
$$ \| \Sample_\Omega^{-1} f \|_{L^p(\A_\Z^d;H)} = \exp( O(d + \rho^{-1} \Log(\max(p,p') \Log \rho^{-1}) ) ) \| f \|_{\ell^p(\Z^d;H)} $$
for all $f \in \ell^p(\Z^d;H)^{\pi(\Omega)}$.
\end{corollary}

\begin{proof}  Set $r$ to be the first natural number such that $(2r)' \leq p \leq 2r$, then $r \sim \max(p,p')$.  Let $k \coloneqq \lfloor \frac{2}{\rho}\rfloor + 1 \sim \rho^{-1}$, and let $S$ be the set constructed by Lemma \ref{iwac}, then $\Omega \subseteq [-\eps,\eps]^d \times \Sigma_{\leq k}$.  If the constant $C$ in \eqref{eps-small} is large enough, Lemma \ref{iwac}(iii) ensures that $(d,k,S,\eps)$ is $(r,\frac{1}{2})$-good, and the claim now follows from Theorem \ref{main-sampling}.
\end{proof}

\appendix

\section{Abstract harmonic analysis}\label{adele-sec}

We define the Pontryagin dual pairs $(\G,\G^*)$ of LCA groups used in this paper.

\begin{itemize}
    \item [(i)]  If $\G = \R$ with Lebesgue measure $\mu_\R = dx$, then $\G^* = \R^* = \R$ with Lebesgue measure $\mu_{\R^*} = d\xi$ is a Pontryagin dual, with pairing $x \cdot \xi \coloneqq x \xi \mod 1$.
    \item[(ii)]  If $\G = \Z$ with counting measure $\mu_\Z$, then $\G^* = \T$ with Lebesgue measure $\mu_{\T} = d\xi$ is a Pontryagin dual, with pairing $x \cdot \xi \coloneqq x \xi$.
    \item[(iii)]  If $\G = \Z/Q\Z$ is a cyclic group for some $Q \in \Z_+$ with normalized counting measure $\int_{\Z/Q\Z} f(x)\ d\mu_{\Z/Q\Z}(x) \coloneqq \E_{x \in \Z/Q\Z} f(x)$, then the \emph{dual cyclic group} $\G^* = \T[Q] = \frac{1}{Q} \Z/\Z$ with counting measure $\mu_{\T[Q]}$ is a Pontryagin dual, with pairing $x \cdot \xi \coloneqq x \xi$.
    \item[(iv)]  If $\G = \hat \Z \coloneqq \varprojlim \Z/Q\Z$ is the compact group of profinite integers with Haar probability measure (using the projection maps from $\Z/Q\Z$ to $\Z/q\Z$ whenever $q$ divides $Q$), then the discrete group $\G^* = \hat \Z^* =  \Q/\Z$ of ``arithmetic frequencies'' with counting measure $\mu_{\Q/\Z}$ is a Pontragin dual, with pairing $x \cdot (\frac{a}{q} \mod 1) \coloneqq \frac{xa \mod q}{q}$.
    \item[(v)]  If $\G_1,\G_2$ are LCA groups with Pontryagin duals $\G_1^*, \G_2^*$, then the product $\G_1 \times \G_2$ (with product Haar measure) is an LCA group with Pontryagin dual $\G_1^* \times \G_2^*$ and pairing $(x_1,x_2) \cdot(\xi_1,\xi_2) \coloneqq x_1 \cdot \xi_1 + x_2 \cdot \xi_2$.  In particular, if $\G = \A_\Z^d= \R^d \times \hat \Z^d$ is the $d^{\mathrm{th}}$ power of the adelic integers\footnote{The adelic integers $\A_\Z$ should not be confused with the larger ring $\A_\Q = \A_\Z \otimes_\Z \Q$ of adelic numbers, which we will not use in this paper.} $\A_\Z \coloneqq \R \times \hat \Z$ (with the product Haar measure $\mu_{\A_\Z^d} \coloneqq \mu^d_\R \times \mu^d_{\hat \Z}$), then \emph{adelic frequency space} $\G^* = (\A_\Z^d)^* = \R^d \times (\Q/\Z)^d$ is a Pontryagin dual (with product measure $\mu_{\R \times \Q/\Z} \coloneqq \mu_\R^d \times \mu_{\Q/\Z}^d$ and the indicated pairing.  
\end{itemize}

More explicitly: an element of $\A_\Z^d$ is of the form $(x,y)$, where $x = (x_1,\dots,x_d) \in \R^d$ and $y = (y_1,\dots,y_d) \in \hat \Z^d$, thus $y_d \mod Q$ is an element of $\Z/Q\Z$ for any positive integer $Q$ (with the compatibility conditions $y_d \mod q = (y_d \mod Q) \mod q$ whenever $q$ divides $Q$), and if $(\xi, \eta) = (\xi_1,\dots,\xi_d, \frac{a_1}{q} \mod 1, \dots, \frac{a_d}{q} \mod 1)$ is an element of the dual group $\R^d \times (\Q/\Z)^d$, then
\begin{align*}
(x,y) \cdot (\xi, \eta) &= x \cdot \xi + y \cdot \eta \\
&= x_1 \xi_1 + \dots + x_d \xi_d + \frac{a_1 y_1 \mod q + \dots + a_d y_d \mod q}{q}.
\end{align*}
We have the canonical inclusion $\iota \colon \Z^d \to \A_\Z^d$ defined by
$$ \iota(n) \coloneqq (n, (n \mod Q)_{Q \in \Z_+})$$
and the projection map $\pi \colon \R^d \times (\Q/\Z)^d \to \T^d$ defined by
$$ \pi(\theta,\alpha) \coloneqq \alpha + \theta;$$
the two maps enjoy the Fourier adjoint relationship
$$ n \cdot \pi(\theta,\alpha) = \iota(n) \cdot (\theta,\alpha)$$
for all $n \in \Z^d$ and $(\theta,\alpha) \in \R^d \times (\Q/\Z)^d$.

We define the following Schwartz-Bruhat spaces $\Schwartz(\G)$ on various LCA groups\footnote{For a definition of Schwartz-Bruhat spaces on arbitrary LCA groups, see \cite{bruhat}, \cite{osb}.} $\G$:

\begin{itemize}
    \item[(i)]  $\Schwartz(\R^d)$ is the space of Schwartz functions on $\R^d$.
    \item[(ii)]  $\Schwartz(\Z^d)$ is the space of rapidly decreasing functions on $\Z^d$, and $\Schwartz(\T^d)$ is the space of smooth functions on $\T$.
    \item[(iii)]  $\Schwartz(\hat \Z^d)$ is the space of locally constant functions $f$ on $\hat \Z^d$, or equivalently those functions of the form $f(x) = f_Q(x \mod Q)$ for some $Q \in \Z_+$ and $f_Q \colon (\Z/Q\Z)^d \to \C$.  $\Schwartz((\Q/\Z)^d)$ is the space of finitely supported functions on $(\Q/\Z)^d$.
    \item[(iv)]  $\Schwartz(\A^d_\Z)$ is the space of functions of the form $f(x,y) = f_Q(x,y \mod Q)$ for some $Q \in \Z_+$ and $f_Q \colon \R^d \times (\Z/Q\Z)^d$ that is Schwartz in the $\R$ variable.  $\Schwartz(\R^d \times (\Q/\Z)^d)$ is the space of functions supported on $\R^d \times \Sigma$ for some finite set $\Sigma \subset (\Q/\Z)^d$ and Schwartz in the $\R^d$ variable.
\end{itemize}

If $\G = \R^d, \Z^d, \A^d_\Z$, one can verify that the Fourier transform ${\mathcal F}_\G$ is a linear isomorphism between $\Schwartz(\G)$ and $\Schwartz(\G^*)$.

See \cite[\S 4]{KMT} for a further development of abstract harmonic analysis and a discussion of the relationship between the integers $\Z$, the adelic integers $\A_\Z$, and other related LCA groups.

\end{document}